%% file: article.tex
\documentclass[hidelinks,onefignum,onetabnum]{siamart251104}

\input{abbrev}
\input{shared}

\usepackage{nicefrac}

\newcommand{\reg}{\text{\scriptsize{reg}}}
\newcommand{\true}{\text{\scriptsize{true}}}
\newcommand{\LS}{\text{\scriptsize{LS}}}

\ifpdf
\hypersetup{
  pdftitle={Sketched FLSQR and FLSMR},
  pdfauthor={A. Bucci, S. Gazzola, and L. Robol}
}
\fi

\begin{document}

\maketitle

\begin{abstract}
LSQR and LSMR are iterative methods, based on the Golub–Kahan bidiagonalization algorithm, widely used for large-scale linear least squares problems. FLSQR and FLSMR are flexible variants of LSQR and LSMR, respectively, based on a flexible Golub-Kahan (Arnoldi-like) factorization algorithm, which naturally allow modifications of the solution approximation subspace and/or handling inexact matrix-vector multiplications with the (transpose of the) coefficient matrix, thereby enabling to enforce prior information into the computed solution. The goal of this paper is to introduce sFLSQR and sFLSMR, i.e., sketched variants of FLSQR and FLSMR, respectively, where randomization becomes particularly effective, as it allows to recover short recurrences for the solution approximation. 
In particular, this paper explores applications to large-scale inverse problems, showing the ability of the new randomized solvers to alleviate computational bottlenecks while preserving reconstruction quality. A theoretical analysis of sFLSQR and sFLSMR is provided, and their performance is validated through numerical experiments.
\end{abstract}



\section{Introduction}\label{sect:intro} 
In the past few decades, randomized numerical linear algebra has emerged as a powerful tool to perform many numerical linear algebra tasks more efficiently and with strong theoretical guarantees, affecting many other scientific computing tasks and applications of computational mathematics; see \cite{hmt, martinsson2020randomized} and references therein. The focus of this work is the interplay of large scale least squares problems, (flexible) Krylov methods, and randomized sketching, with applications to large-scale linear inverse problems in imaging. To the best of our knowledge, randomized Krylov methods for regularizing linear inverse problems have only been considered very recently in \cite{chung2025randomized, sabate2025randomized, sabate2025randomizednew}.

\subsection{Solving linear inverse problems}
Linear discrete inverse problems typically read as 
\begin{equation}
\label{eq:inverseproblem}
\bfb = \bfA \bfx_\true + \bfe\,,
\end{equation}
where $\bfb\in \bbR^m$ collects observed measurements, $\bfA \in \bbR^{m\times n}$ represents a discretized forward model, $\bfx_\true \in \bbR^n$ is an unknown quantity of interest, 
and $\bfe\in \bbR^m$ contains unknown noise or errors in the data.  Throughout this work, 
we assume that $m$ and $n$ are both large, and that $\bfA$ is full rank (but typically ill-conditioned and with singular values rapidly decaying to zero). We also assume that $\bfe$ is a realization of a white Gaussian noise vector. Given $\bfb$ and $\bfA,$ the goal of inverse problems is to approximate $\bfx_\true$, immediately translating into the task of solving a linear least squares problem of the form
\begin{equation}
\label{eq:LS}
\min_{\bfx\in\bbR^n} \| \bfA \bfx - \bfb \|_2\,.
\end{equation}
The minimizer of \eqref{eq:LS} can be analytically expressed as $\bfx_\LS = \bfA^\dagger \bfb = 
   \bfx_\true + \bfA^\dagger \bfe$. 
Because of the large 
condition number $\kappa_2(\bfA) := \sigma_1(\bfA) / \sigma_n(\bfA)$ of $\bfA$, given in terms of the ratio between the largest and smallest singular values of $\bfA$,
\[
\| \bfx_\true - \bfx_\LS \|_2 = \| \bfA^\dagger \bfe \|_2 \gg \| \bfx_\true \|_2,
\]
which makes computing
$\bfx_\LS$ often useless in practice. 

To mitigate the contamination introduced in $\bfx_\LS$ by the so-called ``inverted noise'' $\bfA^\dagger\bfe$, one typically resorts to regularization techniques, which replace the original ill-posed least squares problem \eqref{eq:LS} by a related one that is more robust with respect to perturbations in the data. This is ideally done by encoding some available information about $\bfx_\true$ within the new problem formulation, leading to various regularization techniques. 
Since, in general, in this paper we consider a large-scale and unstructured $\bfA$, the original problem \eqref{eq:LS} is solved by an iterative solver solely relying on 
matrix--vector products with $\bfA$ and $\bfA^{\!\top}$, such as fixed-point 
iterations or Krylov methods (including LSQR \cite{paige1982lsqr} and LSMR \cite{LSMR}). 
In this framework, a common approach to recover a regularized version $\bfx_\reg$ of $\bfx_\LS$ is to terminate the iterations of any such solvers early; see  \cite{hansen2010discrete}. This practice introduces an implicit form of regularization into \eqref{eq:LS}, as the early iterations primarily capture desirable 
information about $\bfx_\true$ from the dominant singular vectors, while the later iterations start converging to $\bfx_\LS$, recovering its unwanted noisy components. This behavior is usually referred to as ``semiconvergence''. Therefore, the number of iterations acts as a regularization parameter specifying the amount of regularization, and effective stopping criteria should act as regularization parameter choice strategies. Among the latter, assuming that an accurate estimate of the noise magnitude $\deltae =\|\bfe\|_2$ in the data is available, the popular discrepancy principle prescribes to select the regularization parameter such that $\|\bfA\bfx_\reg - \bfb\|_2\simeq \deltae$ that, for iterative methods, translates to stopping at the $k$th iteration with residual such that
\begin{equation}\label{eq: discrP}
    \|\bfA\bfx_k-\bfb\|_2\leq \eta\deltae,\quad\mbox{where $\eta>1$ (typically $\simeq 1$) is a safety factor}.
\end{equation}
The relative magnitude of the noise is referred to as noise level $\delta$, i.e., $\delta\!=\!\deltae/\|\bfA\bfx_\true\|_2$.

\subsection{Flexible Krylov methods to enforce additional  regularization}
Flexible Krylov methods (including FLSQR and FLSMR \cite{chung2019}) can be briefly described as Krylov methods 
where the matrix $\bfA$ and/or $\bfA^{\!\top}$ used to build the Krylov subspace can change 
at each step. They were originally introduced to deal with 
non-stationary preconditioning, but can be used to describe 
a variety of phenomena, such as inexact application of the 
operator $\bfA$ and/or $\bfA^{\!\top}$. When treating linear inverse problems, they naturally appear within the framework of variational regularization methods, whereby one penalizes the so-called fit-to-data term \eqref{eq:LS} by adding a regularization term of the form $\lambda R(\bfx)$, where $\lambda\geq 0$ is the so-called regularization parameter. For specific relevant choices of $R(\bfx)$, such as the $\ell_p$ (semi)norm of $\bfx$ ($0<p\leq 1$) or the nuclear norm, one can solve the resulting problem using efficient variations of the classical iteratively reweighted least squares algorithm, whereby the inverses of the weights used to approximate the (semi)norm at hand formally appear as variable preconditioners on the right of $\bfA$; see \cite{survey, gazzola-lowrank, FKSIRW} for more details.

In this work, we are interested in building upon the foundation
given in \cite{gazzola-lowrank} for image reconstruction tasks such as 
deblurring and inpainting (see also the motivating illustration below). Here, the vectors $\bfx$ and $\bfb$ are vectorized 
2D images, assumed square for simplicity: the former is the original one, and the latter is the 
blurred one, possibly also affected by missing data (treated as zero entries). A property that can often be found 
in images represented in matrix form 
is a decay in their singular values. 
%
To enforce the same behavior (low-rankness) in the solution, in \cite{gazzola-lowrank} the authors propose to 
modify the basis vectors (reshaped as 2D arrays with the same shape as the original
image) by applying a truncated SVD to project them onto 
a low-rank manifold. More specifically, if a vector $\bfc$ is the vectorization of a 2D square array $\bfC$, i.e., $\bfc=\vect(\bfC)$ and $\bfC=\vect^{-1}(\bfc)$, 
we define the rank-$r$ truncation operator
\begin{equation}\label{def:lowr-trc}
\tau_r(\bfc)=\vect(\bfC_r)=\vect(\bfU^{\bfC}_r\bfSigma^{\bfC}_r(\bfV^{\bfC}_r)^{\!\top}),
\end{equation}
where $\bfC_r$ denotes the rank-$r$ truncated SVD of $\bfC$ (defined with respect to the truncated singular vectors and values matrices $\bfU^{\bfC}_r,\, \bfSigma^{\bfC}_r,\,\bfV^{\bfC}_r$ of $\bfC$). 
This helps in recovering solutions with 
decaying singular values, since the final approximation is obtained 
by a linear combination with a few of such low-rank basis vectors.

Aside from deblurring and inpainting, flexible Krylov methods may be naturally applied to X-ray tomographic reconstruction problems. Indeed, in common software toolboxes for tomography, such as ASTRA \cite{astra}, different discretization schemes and different model approximations for $\bfA$ and $\bfA^{\!\top}$ are adopted to reduce the computational effort, resulting in an unmatched transpose $\bfA^\sharp\simeq\bfA^{\!\top}$. 
In this work we also consider the application of flexible Krylov methods for handling inexact matrix-vector multiplications with $\bfA^{\!\top}$, providing an alternative to other solvers devised and used for the same purpose in \cite{CT2,CT1}. 

There is, however, a considerable drawback in switching from Krylov methods for 
symmetric problems (implicitly applied to the normal equations associated to \eqref{eq:LS}) to their flexible variants: the 
short recurrence relations that are available in the former 
are lost in the latter. Hence, such flexible approaches come at an
increased computational cost and additional storage requirements. 
This observation inspired 
this work, and the investigation of 
randomized sketching to alleviate this issue. 

\paragraph{Motivating illustration}
We now show an example that demonstrates the potential
benefits of regularization by combining FLSQR with 
(in this case) low-rank truncation. We take the
\texttt{house} test image of size $512\times 512$ pixel, and blur it with a Gaussian PSF with variance approximately $0.25$. 
Then, we subsample it by dropping
some entries, and add $5\%$ of  
white Gaussian noise. 
We finally setup the recovery problem as a linear least squares problem \eqref{eq:inverseproblem}, where
the linear operator $\bfA$ is the blurring composed with subsampling. 
This is solved by 50 
iterations of LSQR and FLSQR, the latter equipped with low-rank truncation (as defined in \eqref{def:lowr-trc}, with $r=30$). 
\begin{figure}
    \centering
    \includegraphics[width=0.22\linewidth, trim = {860px 0 680px 0}, clip]{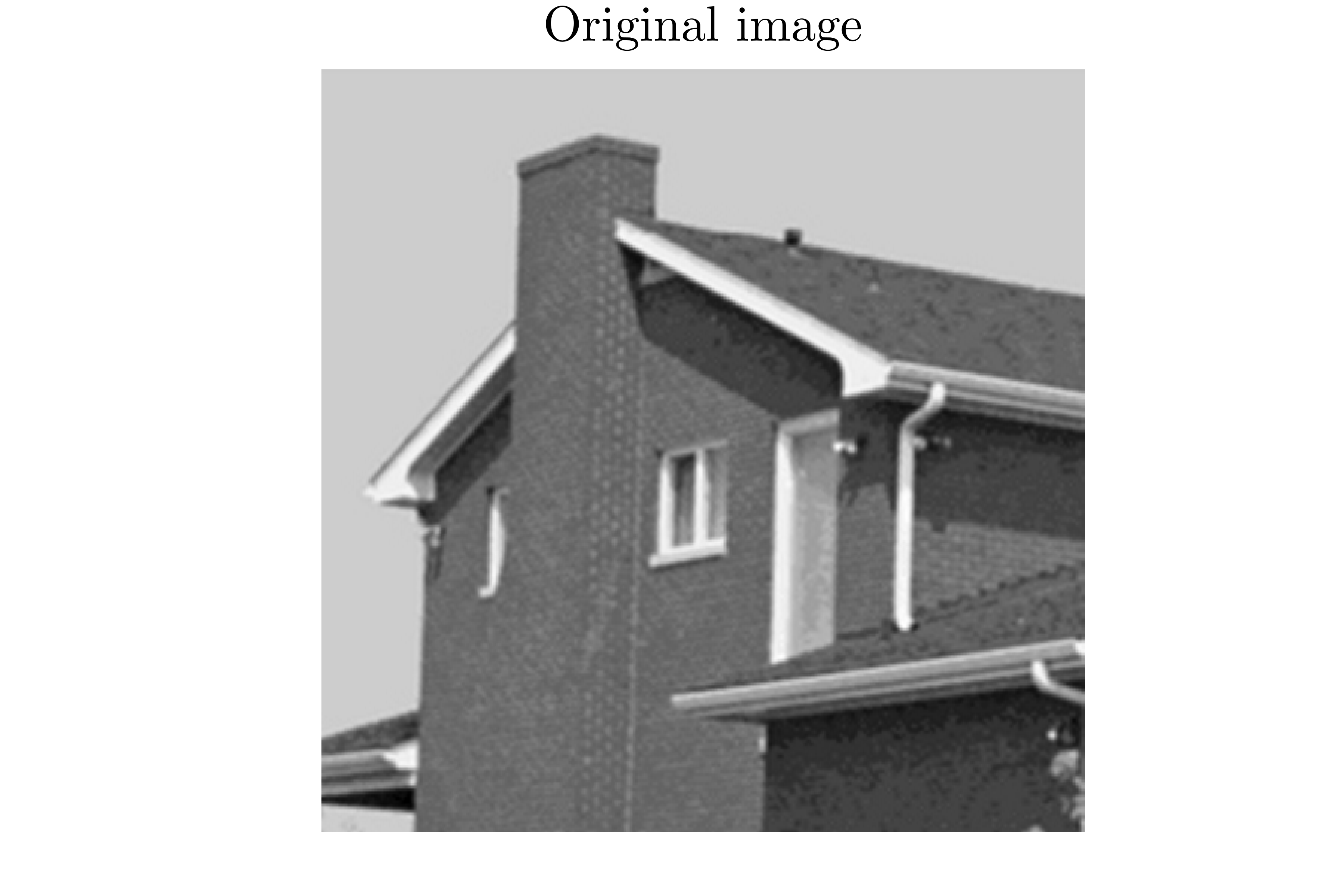}
    \includegraphics[width=0.22\linewidth, trim = {860px 0 680px 0}, clip]{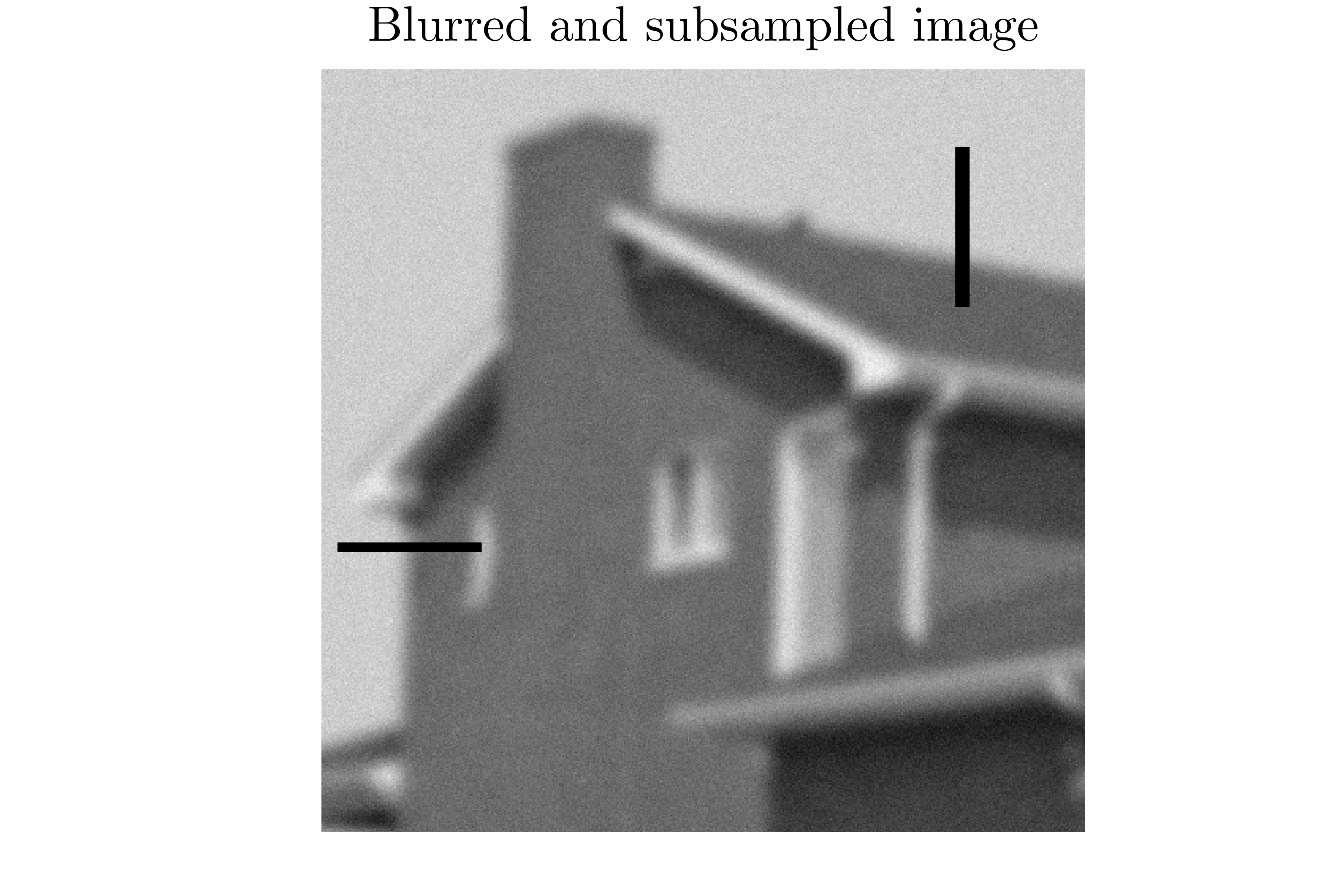}
    \includegraphics[width=0.22\linewidth, trim = {860px 0 680px 0}, clip]{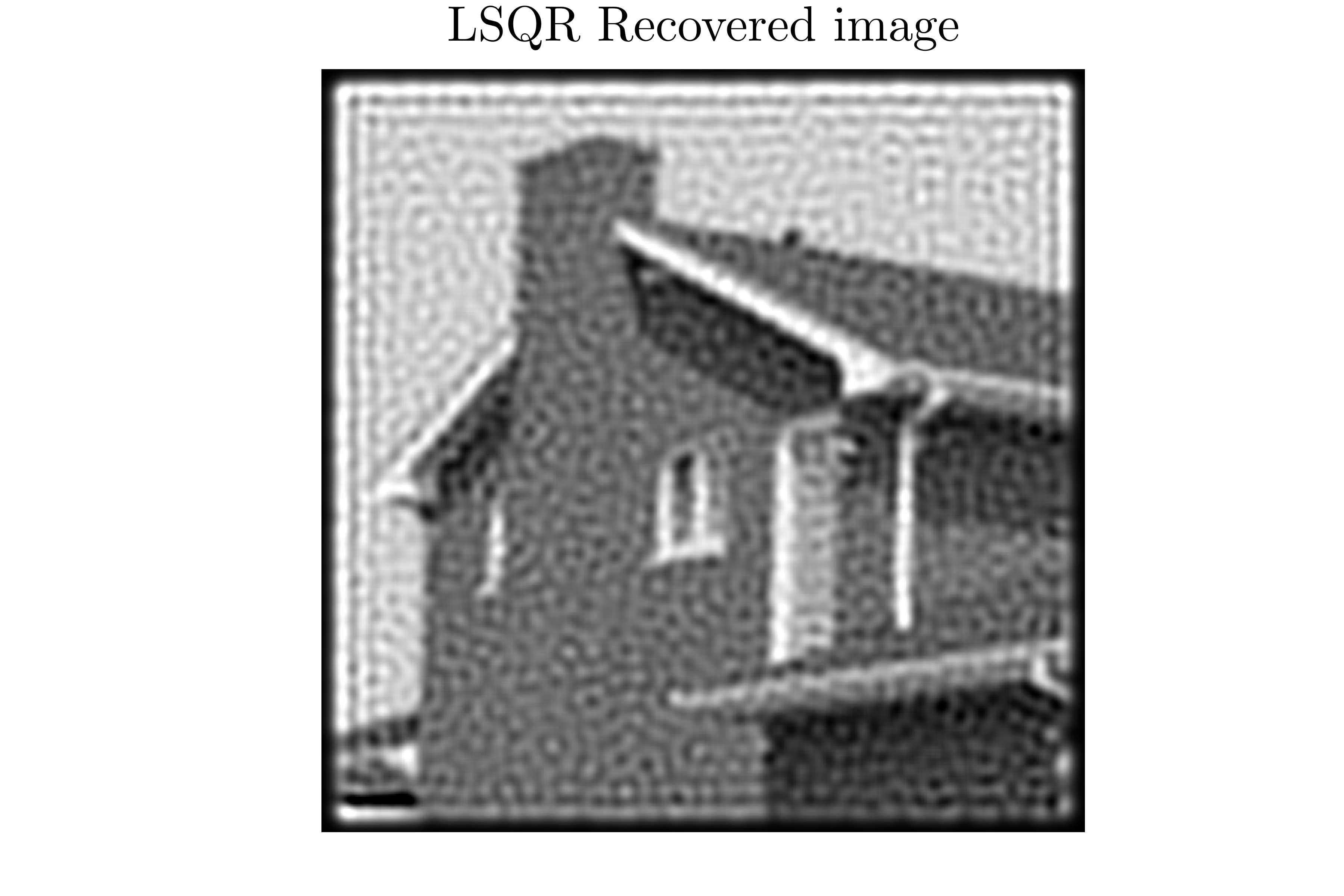}
    \includegraphics[width=0.22\linewidth, trim = {860px 0 680px 0}, clip]{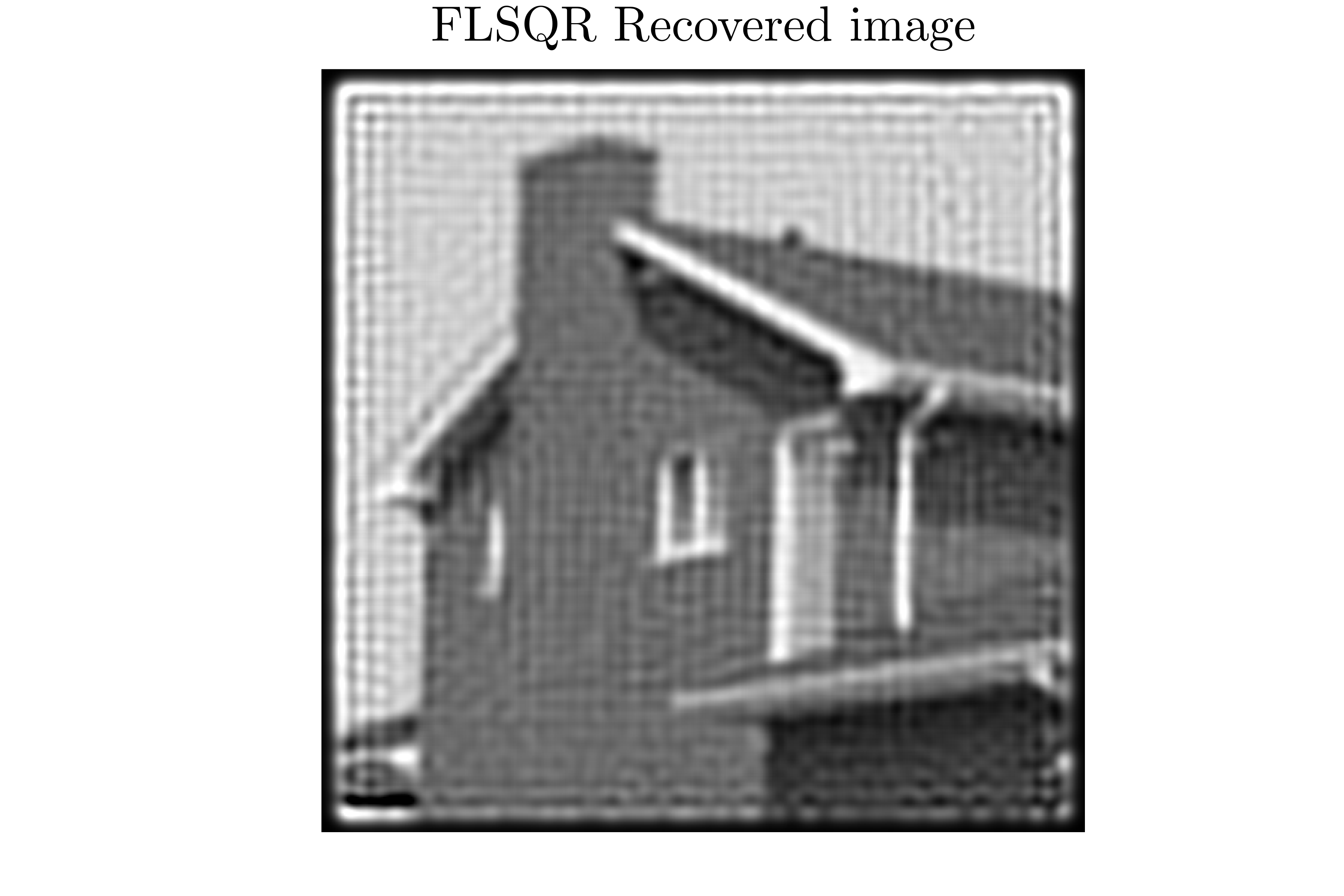}
    \caption{Deblurring and inpainting for the \texttt{house} test image, 
    contaminated with 5\% Gaussian  noise. From the left, we have: 
    the original image, the blurred and subsampled image, the recovery with 
    $31$ iterations of LSQR (which yields the minimum error result among the first $50$ iterations), the recovery with $50$ iterations of FLSQR with low-rank truncation at rank $30$ at each step. }
    \label{fig:example-intro}
\end{figure}
In Fig.~\ref{fig:example-intro} we show the original, corrupted, and 
recovered images. 
For both methods, we report the recovered image
which yields the lowest error among all 
iterations. This is attained 
at the $31$st LSQR iteration, 
and at the $50$th FLSQR iteration. 
The LSQR error is worse than FLSQR error, 
as clearly visible from the error plot in Fig.~\ref{fig:example-intro-plot} (right).
The low-rank truncation in FLSQR improves the accuracy of the reconstruction,
and avoids (or delays) 
semiconvergence. This can be seen again from the error plot in Fig.~\ref{fig:example-intro-plot} (right), where we compare the 
relative 2-norm residuals and errors in the two algorithms. 
%
%
%
\begin{figure}
    \centering
%
    \begin{tikzpicture}
    \begin{groupplot}[
        group style={
            group size=2 by 1,
            horizontal sep=1.5cm,
        }, 
        yticklabel style={
            /pgf/number format/fixed,
            /pgf/number format/precision=2,
        },
        tick label style={font=\small},
        label style={font=\small},
        title style={font=\small},
        width=6.8cm,
        height=5cm,
        xlabel={Iterations},
        grid=both,
        legend style={
            at={(0.95,0.95)},
            anchor=north east,
            legend columns=2,
        },
    ]
    \nextgroupplot[
        title={Relative Residual}
    ]
    \addplot+[very thick,
        blue!75!white, mark = *, mark size = 1.5pt, mark repeat = 4]
    table[
        x=iter_lsqr,
        y=res_lsqr,
        col sep=space
    ]{residuals_flsqr.dat};
    \addlegendentry{LSQR}
    \addplot+[very thick,
        red!75!white, mark = x, mark size = 1.5pt, mark repeat = 4]
    table[
        x=iter_flsqr,
        y=res_flsqr,
        col sep=space
    ]{residuals_flsqr.dat};
    \addlegendentry{FLSQR}
    \nextgroupplot[
        title={Relative Error},
    ]
    \addplot+[very thick, blue!75!white, mark = *, mark size = 1.5pt, mark repeat = 4]
    table[
        x=iter_lsqr,
        y=err_lsqr,
        col sep=space
    ]{errors_flsqr.dat};
    \addlegendentry{LSQR}
    \addplot+[very thick, red!75!white, mark = x, mark size = 1.5pt, mark repeat = 4]
    table[
        x=iter_flsqr,
        y=err_flsqr,
        col sep=space
    ]{errors_flsqr.dat};
    \addlegendentry{FLSQR}
    \end{groupplot}
    \end{tikzpicture}
    \caption{Deblurring and inpainting for the \texttt{house} test image. 2-norm residual and error plot, using LSQR and FLSQR (with low-rank-truncation as in Fig.~\ref{fig:example-intro}).}
    \label{fig:example-intro-plot}
\end{figure}
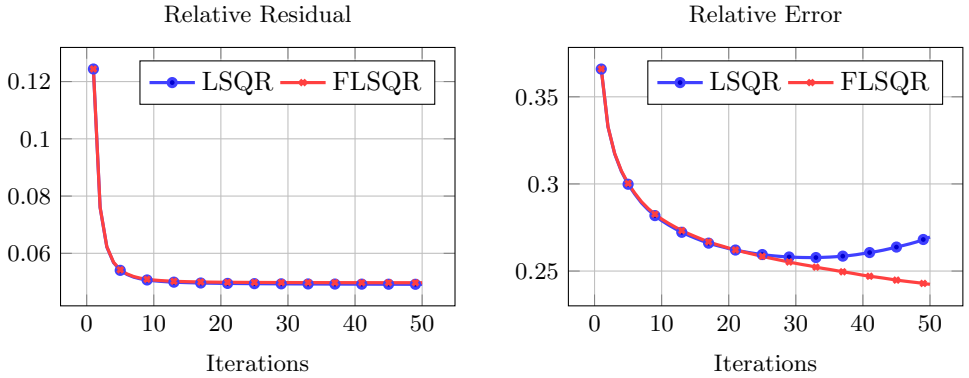
As expected, the residuals are small in both cases, whereas LSQR exhibits semiconvergence around iteration $30$, while FLSQR manages to achieve 
a smaller recovery error. These results are in line with what has already 
been observed in \cite{gazzola-lowrank}. However, the increased accuracy of the flexible method comes at a non-negligible cost: running LSQR requires around 27 seconds, whereas 
FLSQR takes 35 seconds. In this work, we show that the FLSQR cost can be 
lowered, as we can use sketching to run the more effective FLSQR 
at about the same cost
than the standard LSQR method. 

\subsection{Randomized sketching in least squares and inverse problems}
Motivated by recent developments in randomized numerical linear algebra \cite{hmt, mahoney2011randomized, woodruff2014sketching}, and in particular by sketching approaches tailored to large-scale iterative solvers \cite{balabanov2022randomized, nakatsukasa2024fast}, we explore randomized techniques for FLSQR and FLSMR, with the goal of preserving the typical increased reconstruction accuracy achieved by these methods with respect to their standard counterparts (as illustrated above), while avoiding
the computational drawback of having to perform costly full orthogonalization procedures for the basis vectors rather than short-recurrences. 

When considering Arnoldi-based solvers like GMRES for square $\bfA$, the authors of \cite{balabanov2022randomized} propose to perform the (full) orthogonalization on a sketched, low-dimensional version of the Krylov subspace basis, making it significantly cheaper. In a similar setting, the authors of \cite{nakatsukasa2024fast} show that, with appropriate randomization within the sketch-to-solve paradigm, a truncated orthogonalization process for the Krylov approximation subspace basis can still yield high-quality solutions. Our construction is more closely aligned with the last approach. The authors of \cite{chung2025randomized}, targeting inverse problems and  
general rectangular matrices $\bfA$, extend the sketched inner product approach of \cite{balabanov2022randomized} introducing a randomized Golub-Kahan factorization that involves (full) orthogonalization of the sketched, low-dimensional basis vectors. In the inverse problem setting, the authors of \cite{sabate2025randomizednew} apply randomized numerical linear algebra techniques to flexible Krylov solvers based on both the flexible Arnoldi and flexible Golub-Kahan factorizations. Even if \cite{sabate2025randomizednew} also promotes a sketch-to-solve
approach applied to a projected problem,  only minimal-residual methods are considered, so that FLSMR-like solves are not explored; also, strategies for overcoming costs in the storage of the basis vectors (needed for adaptively setting the regularization parameter in Tikhonov problems) are proposed. 

Another line of work that has recently been pursued to reduce the orthogonalization cost in the context of 
inverse problems relies on the changing minimal residual Hessenberg method (CMRH) \cite{sadok1999cmrh}; see \cite{brown2025inner, brown2025h, sabate2025randomized}.
Among these contributions, \cite{sabate2025randomized} is perhaps the closest to ours, as it proposes to use a project and sketch 
paradigm to compute solutions whose residual norms are very close to be optimal, 
as opposed to the ones associated to inner-product-free solvers such CMRH (for square $\bfA$) and LSLU (for general rectangular $\bfA$), which are only quasi-minimal.
Again, the approach in \cite{sabate2025randomized} is different from ours, as the basis for the Krylov subspaces are generated in a fundamentally different manner, and no LSMR-like methods are considered. 
As we will show in the following, (F)LSMR offers distinctive advantages over (F)LSQR when combined with sketching for large noise problems.

\subsection{Main contributions}

We introduce two new sketched flexible 
least squares solvers, 
\textit{sFLSQR} and \textit{sFLSMR}, that are randomized versions 
of the 
FLSQR
and the 
FLSMR
methods, respectively. 
The key idea is to rely on flexibility to 
incorporate any structure in the solution as a form 
of regularization, and at the same time increase 
computational efficiency by using randomized sketching techniques. 
Apart from deriving the new sFLSQR and sFLSMR solvers, the main contributions in the paper are the following:
\begin{description}
\item \textit{Studying the influence of the noise level on sFLSQR and sFLSMR.}
We show that sketching the residual minimization problem associated with FLSQR can lead to noticeable deviations in the residual norms when the latter remain non-negligible, 
a situation typical in inverse problems with moderate or large noise levels. We demonstrate that sFLSMR mitigates this 
issue. This leads to 
a practical rule of thumb for selecting the best sketched 
Krylov method: sFLSQR for low-noise levels, and 
sFLSMR when the noise is large. 

\item \textit{Analyzing the sketched residuals approximation errors.} 
We derive deterministic bounds that relate the residual norms produced by the sketched methods to the optimal residual attainable within the same approximation subspace. We provide a probabilistic bound in expectation for the sFLSQR residual when using Gaussian sketchings, that links the sketch size with the accuracy of the residual.

\item \textit{Using unmatched and approximated transposes.}
In our numerical experiments, we show how to 
exploit the new sketched approach to incorporate the use of 
unmatched transposes (for instance, the GPU implementations of backprojections for Computed Tomography (CT) scan 
image reconstruction problems) in Krylov methods. 
We demonstrate that this can be done while maintaining 
a solid theoretical framework, and without 
sacrificing performance. 
\end{description}

\section{Sketched flexible Krylov methods for least squares problems}
\label{sec:sketched-least-squares}
This section describes the sketched variants of FLSQR (sFLSQR) and FLSMR (sFLSMR) proposed in this work. In order to fully understand the derivations underlying the new methods, we first recall the construction of LSQR, LSMR (Section \ref{sect:LSQR_LSMR}), FLSQR and FLSMR (Section \ref{sec:flsqr}) from Golub-Kahan factorizations. The sFLSQR method is then introduced in Section \ref{sec:sflsqr}, while sFLSMR is introduced in Section \ref{sec:sflsmr}, after providing some insight in Section \ref{sec:residual-size}. An analysis of the errors in the residuals introduced by sketching is proposed in Section \ref{sec: bounds}, and some common guidelines for the choice of the sketching operator are recalled in Section \ref{sec:sketching}. 


\subsection{Golub-Kahan bidiagonalization, LSQR and LSMR}\label{sect:LSQR_LSMR}
$k$ steps of the Golub-Kahan bidiagonalization (GKB) algorithm with starting vector $\bfu_1=\nicefrac{\bfb}{\|\bfb\|_2}$ can be compactly written as the following partial matrix factorizations, 
\begin{equation}\label{eq:GKB}
\bfA^{\!\top} \bfU_{k} = \bfV_k \bfB_{k}^{\!\top},\qquad 
\bfA \bfV_k = \bfU_{k+1} \bfB_{k+1,k}\,,
\end{equation}
where $\bfV_k = [\bfv_1 \ldots \bfv_k]\in\bbR^{n\times k}$ and $\bfU_{k}=[\bfu_1 \ldots \bfu_k]\in\bbR^{m\times k}$ have orthogonal columns of unit 2-norm that span the Krylov subspaces $\mathcal K_k(\bfA^{\!\top}\bfA, \bfA^{\!\top} \bfb)$ and $\mathcal K_k(\bfA\bfA^{\!\top}, \bfb)$, respectively, and $\bfB_{k}\in\bbR^{k\times k}$ is lower bidiagonal, with 
$\bfB_{k+1,k}\in\bbR^{(k+1)\times k}$ obtained by removing the last column of $\bfB_{k+1}$. 
A number of Krylov subspace methods can be defined starting from the GKB: these include LSQR and LSMR. 

The classical LSQR method \cite{paige1982lsqr} is mathematically equivalent 
to running the conjugate gradient (CG) method on the normal equations 
$\bfA^{\!\top} \bfA \bfx = \bfA^{\!\top} \bfb$ associated to \eqref{eq:LS}. It is, however, more stable as it does not 
square the condition number of the problem, and is one of the main 
algorithms available for solving large scale linear least squares problems \eqref{eq:LS}. The approximate solution returned by $k$ iterations of LSQR is such that 
\begin{equation} \label{eq:xlsqr}
\begin{split}
  \bfx_k^{\mathrm{LSQR}} := \bfV_k\bfy_k^{\mathrm{LSQR}},\quad\text{where}\quad\bfy_k^{\mathrm{LSQR}}
  &=\arg\min_{\bfy\in\bbR^k} \| 
   \bfb- \bfA\bfV_k \bfy 
  \|_2 \\
  &=\arg\min_{\bfy\in\bbR^k} \| 
    \|\bfb\|_2\bfe_1 - \bfB_{k+1,k} \bfy
  \|_2.
\end{split}    
\end{equation}

The classical LSMR method \cite{LSMR} is mathematically equivalent to running MINRES on the normal equations $\bfA^{\!\top}\bfA\bfx = \bfA^{\!\top}\bfb$ associated to \eqref{eq:LS}. The approximate solution returned by $k$ iterations of LSMR is such that
\begin{equation} \label{eq:xlsmr}
\begin{split}
\bfx_k^{\mathrm{LSMR}} := \bfV_k\bfy_k^{\mathrm{LSMR}},\quad\mbox{where}\quad\bfy_k^{\mathrm{LSMR}}
&=\arg\min_{\bfy\in\bbR^k} \| 
    \bfA^{\!\top}(\bfb-\bfA\bfV_k \bfy)
  \|_2\\ 
  &=\arg\min_{\bfy\in\bbR^k} \| 
    \bfB_{k+1}^{\!\top}(\|\bfb\|_2 \bfe_1-\bfB_{k+1,k}\bfy) 
  \|_2.
\end{split}    
\end{equation}

Each GKB step, 
when implemented without reorthogonalization, only requires a matrix-vector (matvec) product with $\bfA$, a matvec product with $\bfA^{\!\top}$ and two scalar products (to compute 2-norms). Computing $\bfx_k^{\mathrm{LSQR}}$ and $\bfx_k^{\mathrm{LSMR}}$ efficiently (using smart QR factorization updates for $\bfB_{k+1,k}$ and $\bfB_{k+1}$) only requires the storage of 2 vectors of length $m$ and 3 or 4 vectors of length $n$, respectively.

\subsection{Flexible Golub-Kahan factorization, FLSQR and FLSMR}\label{sec:flsqr}
As remarked in the introduction, when solving large-scale linear inverse problems via an iterative regularizing method, one may wish to adaptively modify the basis vectors of the solution subspace by hard-wiring some information coming from prior assumptions on $\bfx_\true$, in order to improve the quality of the approximations so obtained. When doing so within the classical LSQR and LSMR solvers, the approximation subspace for the solution is not a standard Krylov subspace anymore. To formalize this, we assume to be given an operator $\tau: \mathbb R^n \to \mathbb R^n$ and, starting from $\bfp_1=\nicefrac{\bfA^{\!\top}\bfb}{\|\bfA^{\!\top}\bfb\|_2}$, we define 
the modified basis vectors $\bfz_1,\dots,\bfz_k$ as follows:
\begin{equation}\label{eq: modbasis}
  \bfZ_k := \begin{bmatrix}
  \\
    \bfz_1 & \ldots & \bfz_k \\ 
   \\ 
  \end{bmatrix} :=\begin{bmatrix}
  \\
      \tau(\bfp_1) & \ldots & \tau(\bfp_k) \\ 
   \\ 
  \end{bmatrix} =: \tau(\bfP_k), 
\end{equation}
where we commit a small abuse of notation for $\tau(\bfP_k)$, applying the operator column-wise. The matrices $\bfP_k$ and $\bfZ_k$ can be generated using a generalization of the GKB algorithm, called Flexible Golub-Kahan (FGK) factorization; see \cite{chung2019}. $k$ steps of the FGK factorization with starting vector $\bfw_1=\nicefrac{\bfb}{\|\bfb\|_2}$ can be compactly written as the following partial matrix factorizations,
\begin{equation}\label{eq:flsqr-relations}
      \bfA^{\!\top} \bfW_{k} = \bfP_{k} \bfT_{k},\qquad
      \bfA \bfZ_{k} = \bfW_{k+1} \bfH_{k+1,k}\,,
\end{equation}
where $\bfP_k\in\bbR^{n\times k}$ and $\bfW_k\in\bbR^{m\times k}$ have orthonormal columns $\bfp_i$ and $\bfw_i$, $i=1,\dots,k$ (and are in general different from $\bfV_k$ and $\bfU_k$ appearing in \eqref{eq:GKB}, although they play a similar role), $\bfZ_k$ is as 
in \eqref{eq: modbasis},  $\bfT_k \in \mathbb R^{k \times k}$ 
is upper triangular, and \linebreak[4]$\bfH_{k+1,k} \in \mathbb R^{(k+1) \times k}$ is upper Hessenberg. In the special case $\tau(\bfv) = \bfv$, we obtain 
the usual GKB, with $\bfT_k=\bfB_k^{\!\top}$, $\bfH_{k+1,k}=\bfB_{k+1,k}$, 
$\bfP_k = \bfV_k$ and $\bfW_k = \bfU_k$. We emphasise that the appearance of bidiagonal matrices (or, equivalently, the short term recurrence updates for the vectors $\bfv_i$ and $\bfu_i$) in GKB are a consequence of the underlying inner product with $\bfA^{\!\top} \bfA$; when perturbations 
are introduced by choosing $\tau$ different from the identity, then the matrices $\bfT_k$ and $\bfH_{k+1,k}$ both fill up with $\mathcal O(k^2)$ nonzero entries (representing the orthogonalization coefficients for generating the vectors $\bfp_i$ and $\bfw_i$). 

Similarly to the GKB case, one can build flexible versions of LSQR and LSMR starting from the FGK factorization. Specifically, the approximate solution $\bfx_k^{\mathrm{FLSQR}}$ computed at the $k$th iteration of the flexible LSQR (FLSQR) method is such that $\bfx_k^{\mathrm{FLSQR}}\in\mathrm{range}(\bfZ_k)$ and satisfies an optimality property analogous to the LSQR one, i.e.,
\begin{equation}\label{eq:FLSQR}
\begin{split}
   \bfx_k^{\mathrm{FLSQR}} := \bfZ_k\bfy_k^{\mathrm{FLSQR}},\quad\mbox{where}\quad\bfy_k^{\mathrm{FLSQR}}
   &=\arg\min_{\bfy\in\bbR^k} \| 
    \bfb - \bfA\bfZ_k \bfy
  \|_2\\
  &= \arg\min_{\bfy\in\bbR^k} \| \|\bfb\|_2 \bfe_1 -
    \bfH_{k+1,k} \bfy 
  \|_2.
\end{split}
\end{equation}
Hence, even in this case, thanks to the orthogonality of the columns of $\bfW_{k+1}$, $\bfx_k^{\mathrm{FLSQR}}$ can be efficiently recovered by solving a linear least squares problem of size $O(k)$. The approximate solution $\bfx_k^{\mathrm{FLSMR}}$ computed at the $k$th iteration of the flexible LSMR (FLSMR) method is such that $\bfx_k^{\mathrm{FLSMR}}\in\mathrm{range}(\bfZ_k)$ and satisfies an optimality property analogous to the LSMR one, i.e.,
\begin{align*}
   \bfx_k^{\mathrm{FLSMR}} := \bfZ_k\bfy_k^{\mathrm{FLSMR}}\quad\mbox{where}\quad\bfy_k^{\mathrm{FLSMR}}
   &=\arg\min_{\bfy\in\bbR^k} \| 
    \bfA^{\!\top}(\bfb - \bfA\bfZ_k \bfy)
  \|_2\\ 
  &= \arg\min_{\bfy\in\bbR^k} \| \bfT_{k+1}(\|\bfb\|_2 \bfe_1 -
    \bfH_{k+1,k} \bfy) 
  \|_2.
\end{align*}
Note that FLSMR is mathematically equivalent to running FGMRES \cite{Saad1993} to the normal equations $\bfA^{\!\top}\bfA\bfx = \bfA^{\!\top}\bfb$ associated to \eqref{eq:LS}. 

The $k$th FGK step, similarly to GKB, requires a 
matvec product with $\bfA$ and a matvec product with $\bfA\t$; differently from GKB, $2k$ scalar 
products are needed to orthonormalize the $k$th vector against the previous ones. 

\subsection{Sketched FLSQR}
\label{sec:sflsqr}
This section proposes a further modification to FLSQR to make it 
more practical and cheap to run, without impacting its accuracy. 
The resulting method is called \emph{sketched flexible LSQR} (sFLSQR), 
and can be summarized at a high level as follows:
\begin{itemize}
    \item A partial orthogonalization is performed at each step, providing 
      non-\linebreak[4]orthogonal versions $\modP_k$ and $\modW_k$ of 
      $\bfP_k$ and $\bfW_k$, respectively, as defined in Section~\ref{sec:flsqr}.
    \item The basis $\modZ_k$ is obtained as 
    $\modZ_k = \tau(\modP_k)$.\footnote{Note that, when $\tau(\bfv)$ is nonlinear, 
we generally have $\mathrm{range}(\mxZ_k^{(p)}) \neq \mathrm{range}(\mxZ_k)$.}
    \item The minimization problem $\min_{\bfy\in\bbR^k}\| \bfA \modZ_k \bfy - \bfb\|_2$ is solved 
      by a sketching procedure, instead of relying on the Golub-Kahan 
      relations.
\end{itemize}
We now provide a few more details on how to achieve this. First, let us 
observe that even if a partial reorthogonalization is performed at 
each step, we can still write the recurrence relations for a Golub-Kahan-like 
algorithm. 
Since the matrices involved are not the same of Section~\ref{sec:flsqr} and, in general, do not even span the same subspaces, we use a slightly different notation:
\begin{align}\label{eq: FGK}
     \bfA^{\!\top} \modW_{k} &= \modP_{k} \modT_{k}, &
     \bfA \modZ_{k} &= \modW_{k+1} \modH_{k+1,k},&\quad\mbox{with}\quad 
     \modZ_k &= \tau(\modP_k)\,. 
\end{align}
In this context, the matrices $\modT_k$ and $\modH_k$ are upper triangular and Hessenberg, 
but also banded. However the matrices $\modP_k$ and $\modW_k$ are not orthogonal; therefore, in general, for all $\bfy\in\bbR^k$,
\[
  \| \bfA \modZ_k \bfy - \bfb \|_2 = \| \modW_{k+1} \modH_{k+1,k} \bfy - \bfb \|_2 \neq 
    \| \modH_{k+1,k} \bfy - \| \bfb \|_2 \bfe_1 \|_2, 
\]
and there is no immediate way to compute the minimizer of such quantities 
at each step. Note that the residual $\bfA \modZ_k \bfy - \bfb$ belongs to the column span of $[ \bfA\modZ_k, \bfb ]$, 
which has dimension at most $k+1$. Hence, using a $(k+1)$-oblivious $\epsilon$-subspace 
embedding $\bfS\in \mathbb{R}^{s\times m}$ with $\epsilon < 1$, we can seek an approximate solution by minimizing the sketched norm $\| \bfS(\bfA \modZ_k \bfy - \bfb) \|_2$, which satisfies
\begin{equation}\label{sketchProp}
  (1 - \epsilon) \| \bfA \modZ_k \bfy - \bfb \|_2 \leq \| \bfS(\bfA \modZ_k \bfy - \bfb) \|_2 \leq (1 + \epsilon) \| \bfA \modZ_k \bfy - \bfb \|_2 
\end{equation}
for all $\bfy\in\bbR^k$.
We have that 
$\bfS \bfA \modZ_k = [\bfS \bfA \modZ_{k-1}, \bfS \bfA \bfz^{(p)}_k]$ so, after computing $\bfS\bfb$, the matrix $\bfS \bfA \modZ_k$ 
can be computed one column at a time throughout the 
iterations.
The approximate solution returned by $k$ iterations of sFLSQR is such that
\begin{equation}\label{eq:sFLSQR}
    \bfx_k^{\text{sFLSQR}}=\modZ_k \bfy_k^{\text{sFLSQR}},\quad\mbox{where}\quad
    \bfy_k^{\text{sFLSQR}}=\arg\min_{\bfy\in\bbR^k}\| \bfS(\bfA \modZ_k \bfy - \bfb)\|_2
\end{equation}
and satisfies
\begin{equation}\label{res_bound}
\begin{split}
  \| \bfA \modZ_k \bfy_k^{\text{sFLSQR}} - \bfb \|_2 &\leq 
  (1 - \epsilon)^{-1} \| \bfS (\bfA \modZ_k \bfy_k^{\text{sFLSQR}} - \bfb) \|_2 \\ 
  &\leq (1 - \epsilon)^{-1} \| \bfS\bfA \modZ_k \bfy - \bfS \bfb \|_2
  \leq \frac{1 + \epsilon}{1 - \epsilon}\| \bfA\modZ_k \bfy - \bfb \|_2,
  \end{split}
\end{equation}
for every $\bfy\in\bbR^k$. 
In particular, sFLSQR yields a residual that, 
up to a factor $C_\epsilon := \frac{1 + \epsilon}{1 - \epsilon}$, is as small as the optimal one attainable in $\textrm{range}(\modZ_k)$. 
We observe that this is in line with similar ideas recently proposed in 
\cite{chung2025randomized,sabate2025randomized} and in the context of 
Krylov methods with non-orthogonal bases \cite{brown2025inner}.

When the orthogonalization is limited to the last 
$\ell$ vectors, each step of sFLSQR requires: 
a matrix-vector (matvec) product with $\bfA$, a matvec product with 
$\bfA^{\!\top}$, $2\ell+2$ scalar products, and the sketching 
of a vector. The pseudocode 
for sFLSQR is reported in Algorithm~\ref{alg:sflsqr}.
\begin{algorithm}[h]
\caption{sFLSQR}
\label{alg:sflsqr}
\begin{algorithmic}[1]
\Require Matrix $\bfA$, right-hand-side $\bfb$, sketching matrix
$\bfS$, maximum iterations \texttt{maxit}, tolerance \texttt{tol}, orthogonalization window $\ell$, truncation operator $\tau$
\Ensure Approximate solution $\bfx_k$
\State $\beta = \|\bfb\|_2$, \quad $\bfw_1 = \bfb / \beta$
\State $\bfz = \bfA^{\!\top} \bfw_1$
\State $\bfs_{\bfb} = \bfS \bfb$
\State $\bfZ = [\ ]$, \quad $\bfW = [\bfw_1]$, \quad $\bfS_{\bfA\bfZ} = [\ ] $

\For{$k = 1,\dots,\texttt{maxit}$}

    \For{$j = \max(1,k-\ell),\dots,k-1$}
        \State $\bfz = \bfz - \langle \bfz_j,\bfz\rangle\bfz_j$
    \EndFor

    \State $\bfz = \bfz / \|\bfz\|_2$
    \State $\bfz_k = \tau(\bfz)$
    \State $\bfZ=[\bfZ, \bfz_k]$

    \State $\bfw = \bfA \bfz_k$
    \State $\bfS_{\bfA\bfZ} = [\bfS_{\bfA\bfZ}, \bfS \bfw]$

    \For{$j = \max(1,k-\ell),\dots,k$}
        \State $\bfw = \bfw - \langle \bfw_j,\bfw\rangle\bfw_j$
    \EndFor
    \State $\bfw_{k+1} = \bfw/\|\bfw\|_2$
    \State $\bfW = [\bfW, \bfw_{k+1}]$

    \State $\bfz = \bfA^{\!\top} \bfw_{k+1}$

    \State
    $
        \bfy_k
        =
        \arg\min_{\bfy}
        \|
        \bfs_{\bfb} - \bfS_{\bfA \bfZ}\bfy
        \|_2
    $

    \If{$\|\bfs_{\bfb} - \bfS_{\bfA\bfZ}\bfy_k\|_2 < \texttt{tol} \cdot \|\bfs_\bfb\|_2$}
        \State \textbf{break}
    \EndIf

\EndFor

\State $\bfx_k = \bfZ\bfy_k$

\State \Return $\bfx_k$
\end{algorithmic}
\end{algorithm}


\subsection{Dealing with large-noise problems} \label{sec:residual-size}

As we will discuss in more detail in Section~\ref{sec:sketching}, 
there are multiple ways to select an oblivious embedding $\bfS$, which 
often come as a trade-off between theoretical guarantees and performances. 

However, it should be pointed out that the dimension of the sketching 
---which is inherently linked to the cost of computing it--- usually 
scales as $\mathcal O(\epsilon^{-2})$; see \cite{martinsson2020randomized}. Hence, small choices of 
$\epsilon$ are not feasible in practice, and the quasi-optimal constant 
$C_\epsilon = \frac{1 + \epsilon}{1 - \epsilon}$ cannot be ignored for problems where the residual is not small. This is the case for most image reconstruction problems, such as deblurring, inpainting, 
or those arising from tomography. In general, this issue arises 
any time the noise level $\delta$ is large since, according to the discrepancy principle mentioned in Section~\ref{sect:intro}, in order to recover a regularized solution one should stop as soon as the relative residual norm hits $\delta$. Moreover, theoretical studies on the behavior of LSQR for inverse problems show that the relative residual stabilizes around $\delta$ (even when the solver is in an under-regularization regime, i.e., when increasing the number of iterations); see, e.g., \cite{gazzola2015survey, HnPlSt09}. Experimental evidence suggests that this happens for FLSQR, too. We emphasize that the issue of a potentially enlarged sketched residual in the presence of a substantial residual is not specific of FLSQR, but it is rather 
related to the sketched relationship in the minimization 
problem \eqref{eq:sFLSQR}, as detailed in \eqref{res_bound}.

We now build an example to show this behavior, and compare the following solvers:
\begin{itemize}
    \item LSQR: with the standard implementation based on GKB;
    \item sLSQR: a modified version of LSQR, whereby the minimization problem (top rightmost equation in \eqref{eq:xlsqr}) at each iteration is solved via sketching rather than exploiting the usual orthogonality relations from GKB.
\end{itemize}
In this way we can more directly assess the impact of sketching for minimal residual methods applied to ill-posed problems affected by noise.\footnote{Note that, choosing $\tau(\bfv) = \bfv$ in \eqref{eq: modbasis}, the first solver is mathematically equivalent to FLSQR, while the latter is mathematically equivalent to sFLSQR.}
\paragraph{Experimental setup}
\label{par:setup}
\begin{itemize}
    \item We generate a matrix $\bfA$ of size $m \times n$, with 
      $m = 1024$ and $n = m / 2$, with decaying singular values, chosen 
      as 
      \begin{equation}\label{ill_decay_sv}
\rho^{1-i}, \quad i=1,\dots,n,
\end{equation}
and $\rho = 1.01$. The solution 
      $\bfx_\true$ is chosen as the constant vector of ones, 
      and $\bfb_\true$ is computed as $\bfb_\true = \bfA\bfx_\true$; the problem is normalized 
      to have $\| \bfb_\true \|_2 = 1$;
    \item The noisy vector $\bfb$ is computed by adding a white Gaussian random noise vector $\bfe$ to $\bfb_\true$, with 
      $\| \bfe \|_2 = \delta$ and 
      with $\delta \in \{ 0.01, 0.10 \}$; 
      \item The sketching is Gaussian, with 
     $s = 2k + 1$ rows, where $k$ is the maximum number of 
     iterations. 
    \item The solution is recovered with both the standard LSQR method and the sketched 
      counterpart described above;
\end{itemize}

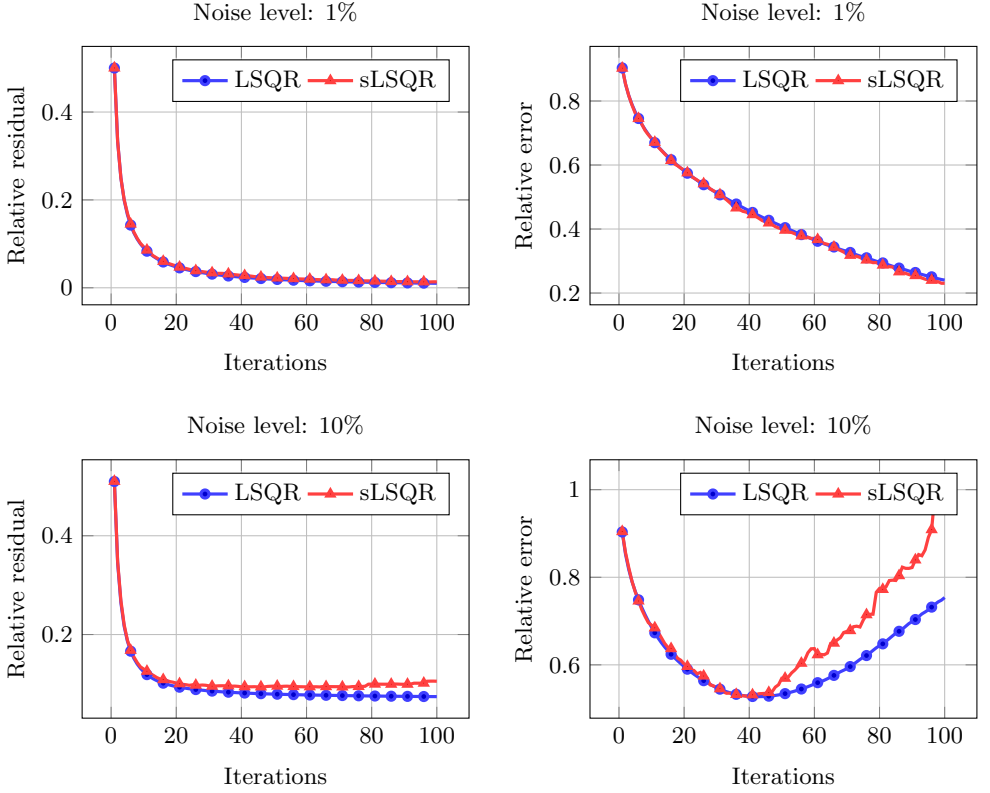
\begin{figure}
    \centering
    \begin{tikzpicture}
    
    \begin{groupplot}[
        group style={
            group size=2 by 2,
            horizontal sep=1.6cm,
            vertical sep=2.05cm,
        },
        ylabel style={yshift=-1em},
        width=6.7cm,
        height=5cm,
        xlabel={Iterations},
        grid=both,
        tick label style={font=\small},
        label style={font=\small},
        title style={font=\small},
        legend style={
            font=\small,
            at={(0.95,0.95)},
            anchor=north east,
            legend columns=2,
        }
    ]
    
    \nextgroupplot[
        title={Noise level: 1\%},
        ylabel = {Relative residual}
    ]
    \addplot+[very thick,
        blue!75!white, mark = *, mark size = 1.5pt, mark repeat = 5]
    table[
        x=iter,
        y=l1_lsqr,
        col sep=space
    ]{paper_01_residuals.dat};
    \addlegendentry{LSQR}
    
    \addplot+[very thick,
        red!75!white, mark = triangle*, mark size = 1.5pt, mark repeat = 5]
    table[
        x=iter,
        y=l1_sflsqr,
        col sep=space
    ]{paper_01_residuals.dat};
    \addlegendentry{sLSQR}
    
    \nextgroupplot[
        title={Noise level: 1\%},
        ylabel = {Relative error}
    ]
    \addplot+[very thick,
        blue!75!white, mark = *, mark size = 1.5pt, mark repeat = 5]
    table[
        x=iter,
        y=l1_lsqr,
        col sep=space
    ]{paper_01_errors.dat};
    \addlegendentry{LSQR}
    
    \addplot+[very thick,
        red!75!white, mark = triangle*, mark size = 1.5pt, mark repeat = 5]
    table[
        x=iter,
        y=l1_sflsqr,
        col sep=space
    ]{paper_01_errors.dat};
    \addlegendentry{sLSQR}
    
    \nextgroupplot[
        title={Noise level: 10\%},
        ylabel = {Relative residual}
    ]
    \addplot+[very thick,
        blue!75!white, mark = *, mark size = 1.5pt, mark repeat = 5]
    table[
        x=iter,
        y=l2_lsqr,
        col sep=space
    ]{paper_01_residuals.dat};
    \addlegendentry{LSQR}
    
    \addplot+[very thick,
        red!75!white, mark = triangle*, mark size = 1.5pt, mark repeat = 5]
    table[
        x=iter,
        y=l2_sflsqr,
        col sep=space
    ]{paper_01_residuals.dat};
    \addlegendentry{sLSQR}
    
    \nextgroupplot[
        title={Noise level: 10\%},
        ylabel = {Relative error}
    ]
    \addplot+[very thick,
        blue!75!white, mark = *, mark size = 1.5pt, mark repeat = 5]
    table[
        x=iter,
        y=l2_lsqr,
        col sep=space
    ]{paper_01_errors.dat};
    \addlegendentry{LSQR}
    
    \addplot+[very thick,
        red!75!white, mark = triangle*, mark size = 1.5pt, mark repeat = 5]
    table[
        x=iter,
        y=l2_sflsqr,
        col sep=space
    ]{paper_01_errors.dat};
    \addlegendentry{sLSQR}
    
    \end{groupplot}
    
    \end{tikzpicture}

    \caption{Artificial test problem described in Section~\ref{sec:residual-size}, with noise levels $1\%$ and $10\%$. Residual and error norms for the solution 
      recovered with LSQR and sLSQR with Gaussian sketching.}
    \label{fig:artificial-01}
\end{figure}

The residuals and the error for this problem
are reported in 
Fig.~\ref{fig:artificial-01}. 
For both cases, the residuals obtained at the end 
of the iterations of sLSQR are around $30\%$
higher than the ones of LSQR; for the test case 
with noise level of $1\%$, this has a limited
impact, since the residual is much smaller; when 
the noise is larger, we see a distinct difference 
in both residuals and norms.
The LSQR and sLSQR error curves for $1\%$ noise level are very 
similar. The corresponding error plots for $10\%$ noise level are more difficult 
to interpret. We have a good match between LSQR and sLSQR in the first 
iterations, and then the two become quite different (although this happens when both LSQR and sLSQR are after the semiconvergence point, and therefore one should in principle have already stopped the solvers). From our experience, 
the results are not always worse: sometimes the sketching gives good results, 
sometimes it does not. It is, however, generally not very reliable. 



\subsection{Sketched FLSMR}
\label{sec:sflsmr}

The tests run with data affected by large noise suggest that sFLSQR can be ineffective for problems affected by noise of large magnitude and large least squares 
residual (obtained by stopping the iterations according to the discrepancy principle). 
Since FLSMR is implicitly solving a least squares
problem whose residual is damped by multiplying 
it with $\mxA^{\!\top}$, we expect its sketched version to have 
better performances. For this reason, we now extend the FLSMR method
discussed in Section~\ref{sec:flsqr} adding a sketching step, providing an alternative to sFLSQR. 

As for sFLSQR, we consider the bases $\modP_k, \modZ_k = \tau(\modP_k), \modW_k$
computed by the FGK algorithm with partial reorthogonalization as 
in \eqref{eq: FGK}. We then sketch the minimization problem associated with
FLSMR, which yields
\begin{equation}\label{eq:sFLSMR}
\begin{split}
\bfx_k^{\text{sFLSMR}}=\modZ_k \bfy_k^{\text{sFLSMR}},\\
\mbox{where}\quad
\bfy_k^{\text{sFLSMR}}&=\arg\min_{\bfy\in\bbR^k}\| \bfS(\bfA^{\!\top}\bfA \modZ_k \bfy - \bfA^{\!\top}\bfb)\|_2\\
 &= 
  \arg\min_{\bfy\in\bbR^k} \| \bfS \bfA^{\!\top} \modW_k \modH_{k+1,k} \bfy - \bfS \bfA^{\!\top} \bfb\|_2 \\ 
 &= \arg\!\min_{\bfy\in\bbR^k}\! \| \bfS \modP_{k+1} (\modT_{k+1} \modH_{k+1,k} \bfy\! -\! \|\bfA^{\!\top} \bfb\|_2\bfe_1)\|_2. 
\end{split}
\end{equation}
Similarly to what previously discussed for sFLSQR, 
solving the least squares problem in \eqref{eq:sFLSMR} only requires the banded matrices $\modT_k, \modH_{k+1,k}$ 
(which are obtained from the (short) recurrence relations \eqref{eq: FGK}) and  
computing $\bfS \modP_k$ (which can be extended one 
column at a time as the iterations proceed). 
%
%
Equivalently, we can compute the columns of 
$\bfS \bfA^{\!\top} \modW_k$ as the iteration proceeds, and 
then solve the least square problem obtained by right 
multiplication by $\modH_{k+1,k}$. The latter is 
the procedure we have implemented in our code. 
The dominant terms in the computational cost are the same as sFLSQR (detailed at the end of Section~\ref{sec:sflsqr}). The pseudocode 
for sFLSMR is reported in Algorithm~\ref{alg:sflsmr}.
\begin{algorithm}[h]
\caption{sFLSMR}
\label{alg:sflsmr}
\begin{algorithmic}[1]
\Require Matrix $\bfA$, right-hand-side $\bfb$, sketching matrix
$\bfS$, maximum iterations \texttt{maxit}, tolerance \texttt{tol}, orthogonalization window $\ell$, truncation operator $\tau$
\Ensure Approximate solution $\bfx_k$
\State $\beta = \|\bfb\|_2$, \quad $\bfw_1 = \bfb / \beta$
\State $\bfz = \bfA^{\!\top} \bfw_1$
\State $\bfs_{\bfA^{\!\top} \bfb} = \beta \cdot \bfS \bfz $
\State $\bfZ = [\ ]$, \quad $\bfW = [\bfw_1]$, \quad $\bfS_{\bfA^{\!\top}\bfW} = [\bfs_{\bfA^{\!\top} \bfb}/\beta ]$
\State $\bfH=[\ ]$ \Comment{corresponding to $\bfH_{k+1,k}^{(p)}$  at the $k$ iteration of \eqref{eq: FGK}}

\For{$k = 1,\dots,\texttt{maxit}$}

    \For{$j = \max(1,k-\ell),\dots,k-1$}
        \State $\bfz = \bfz - \langle \bfz_j,\bfz\rangle\bfz_j$
    \EndFor

    \State $\bfz = \bfz / \|\bfz\|_2$
    \State $\bfz_k = \tau(\bfz)$
    \State $\bfZ=[\bfZ, \bfz_k]$

    \State $\bfw = \bfA \bfz_k$

    \For{$j = \max(1,k-\ell),\dots,k$}
        \State $[\bfH]_{j,k} = \langle \bfw_j,\bfw\rangle$
        \State $\bfw = \bfw - [\bfH]_{j,k}\bfw_j$
    \EndFor
    \State $[\bfH]_{k+1,k} = \|\bfw\|_2$ 
    \State $\bfw_{k+1} = \bfw/[\bfH]_{k+1,k}$
    \State $\bfW = [\bfW, \bfw_{k+1}]$

    \State $\bfz = \bfA^{\!\top} \bfw_{k+1}$ 
    \State 
    $
        \bfy_k
        =
        \arg\min_{\bfy}
        \|
        \bfs_{\bfA^{\!\top} \bfb} - [\bfS_{\bfA^{\!\top}\bfW}, \bfS\bfz][\bfH]_{1:k+1,1:k}\bfy
        \|_2
    $    \If{$\|\bfs_{\bfA^{\!\top}\bfb} - \bfS_{\bfA^{\!\top}\bfW} \bfy_k\|_2 < \texttt{tol} \cdot \|\bfs_{\bfA^{\!\top}\bfb}\|_2$}
        \State \textbf{break}
    \EndIf
    \State $\bfS_{\bfA^{\!\top}\bfW} = [\bfS_{\bfA^{\!\top}\bfW}, \bfS\bfz]$
\EndFor

\State $\bfx_k = \bfZ\bfy_k$

\State \Return $\bfx_k$
\end{algorithmic}
\end{algorithm}

\begin{remark}
    We note that sFLSMR is mathematically equivalent to 
    sketching the minimization problem \eqref{eq:sFLSQR} in sFLSQR with $\bfS \bfA^{\!\top}$
    in place of $\bfS$; indeed, trivially,
    \[
      \|
        \bfS \bfA^{\!\top} (\bfA \modZ_k \bfy - \bfb)
      \|_2 = 
      \|
        \bfS (\bfA^{\!\top} \bfA \modZ_k \bfy - \bfA^{\!\top} \bfb)
      \|_2,
    \]
where the leftmost quantity has to be regarded as sFLSQR with sketching matrix $\bfS\bfA^{\!\top}$, and the rightmost quantity defines sFLSMR with sketching matrix $\bfS$; see \eqref{eq:sFLSMR}. 
This interpretation already offers some insight into the reasons why sFLSMR is expected to provide lower residual norm than sFLSQR, which will be made more precise in Section~\ref{sec: bounds}. 
Indeed, 
taking $\bfU\bfSigma\bfV^{\!\top}=\bfA$ to be the SVD of $\bfA$ and assuming, without loss of generality, that $\|\bfA\|_2=\|\bfSigma\|_2=\sigma_1(\bfA)\leq 1$, we get $\|\bfA^{\!\top}(\bfA \modZ_k \bfy - \bfb)\|_2\leq\|\bfA \modZ_k \bfy - \bfb\|_2$. This trivial estimate can be refined considering the following general facts about linear inverse problems and iterative regularization methods (see \cite{hansen2010discrete}):
\begin{itemize}
\item the singular values of $\bfA$ typically decay quite quickly; 
\item correspondingly, the singular vectors of $\bfA$ display increasing oscillations;
\item the iterations of a solver for \eqref{eq:inverseproblem} may be stopped according to the discrepancy principle as a proxy to establish that the residual $\bfA \modZ_k \bfy - \bfb$ resembles noise. 
\end{itemize}
Therefore, premultiplying by $\bfU^{\!\top}$ the residual vector $\modW_{k+1}(\modH_{k+1,k}\bfy-\|\bfb\|_2\bfe_1)$ results in a vector whose last  entries dominate the first ones; when the resulting vector is rescaled by the singular values of $\bfA$ (i.e., premultiplication by $\bfSigma$), the last dominant entries are reduced, leading to
$\|\bfA^{\!\top}(\bfA \modZ_k \bfy - \bfb)\|_2\ll\|\bfA \modZ_k \bfy - \bfb\|_2$. 
Applying residual bounds analogous to \eqref{res_bound} to the residual premultiplied by $\bfA^{\!\top}$ is therefore likely to lead to tighter bounds for the sketched normal equations residual. The above facts can be made more precise if LSQR (rather than FLSQR with partial orthogonalization) is considered, assuming the so-called discrete Picard condition and a specific decay of the singular values of $\bfA$; see \cite{GazzolaSilvia2016Iotd, gazzola2015survey,  HnPlSt09}.
\end{remark}

\paragraph{Performance on large noise problems}

To demonstrate that sFLSMR is effective for problems with large residuals, we repeat
the artificial experiment of Section~\ref{sec:residual-size}, including the 
results obtained by running sFLSMR along with LSQR and sFLSQR. We recall
that in this context we are just testing 
the use of sketching, 
and there is no use of the ``flexibility'' in the methods; therefore we use the acronym sLSQR in place of sFLSQR (as already done in Section~\ref{sec:residual-size}) and, similarly, we use the acronym sLSMR in place of sFLSMR (to signify that  the minimization problem on the top rightmost equation in \eqref{eq:xlsmr} is solved via sketching rather 
than exploiting the usual orthogonality relations from GKB). The results are reported in Fig.~\ref{fig:artificial-01-sflsmr}. It is immediately
visible that the residual behavior of sFLSMR tracks much more closely that of LSQR on large noise problems. The error behavior is also similar in the first iterations, then has a mild delay, but is much 
more robust to semiconvergence. We emphasize that this is not directly related to the sketching and, even if not shown here, a similar behavior is visible 
when running standard LSMR \cite{ChungPalmer2015}. As expected, LSQR and sFLSQR have a slightly lower error
than sFLSMR for low-noise problems (although the latter error may decrease to a similar value if more iterations are performed); sFLSMR is a better choice when large-noise is present.

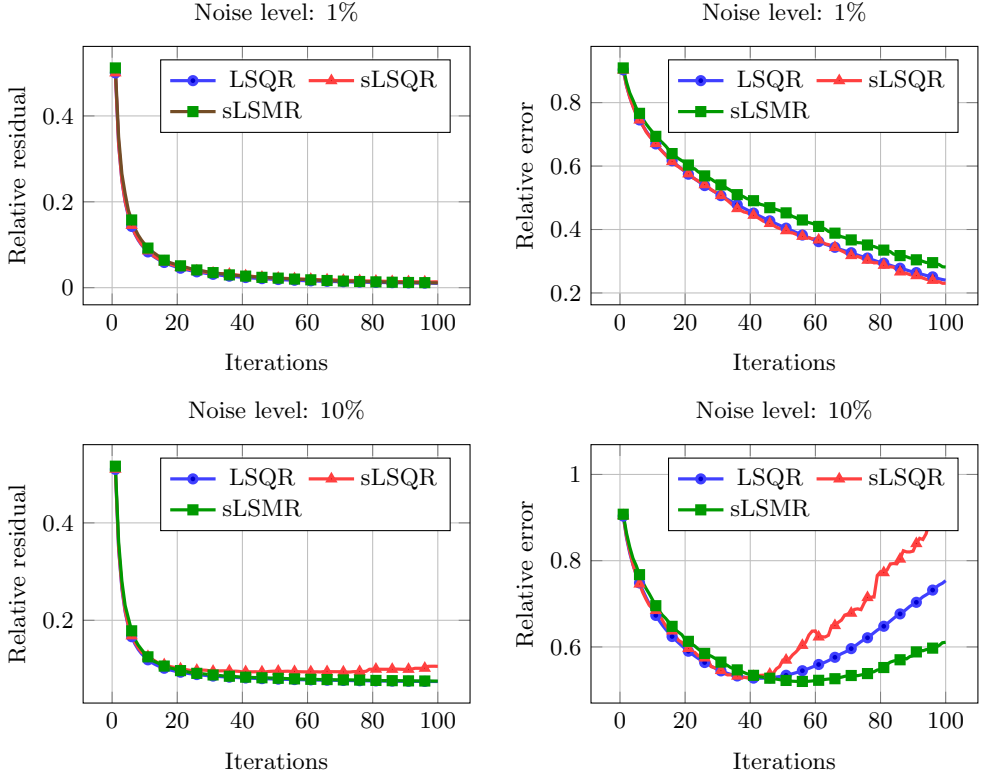
\begin{figure}
    \centering
\begin{tikzpicture}

\begin{groupplot}[
    group style={
        group size=2 by 2,
        horizontal sep=1.6cm,
        vertical sep=1.85cm
    },
    ylabel style={yshift=-1em},
    width=6.7cm,
    height=5cm,
    xlabel={Iterations},
    grid=both,
    tick label style={font=\small},
    label style={font=\small},
    title style={font=\small},
    legend style={
        font=\small,
        at={(0.95,0.95)},
        anchor=north east,
        legend columns=2
    }
]

\nextgroupplot[
    title={Noise level: 1\%},
    ylabel={Relative residual}
]
\addplot+[very thick,
        blue!75!white, mark = *, mark size = 1.5pt, mark repeat = 5]
table[
    x=iter,
    y=l1_lsqr,
    col sep=space
]{paper_01_sflsmr_residuals.dat};
\addlegendentry{LSQR}

\addplot+[mark = triangle*, mark size = 1.5pt, mark repeat = 5, very thick,
        red!75!white]
table[
    x=iter,
    y=l1_sflsqr,
    col sep=space
]{paper_01_sflsmr_residuals.dat};
\addlegendentry{sLSQR}

\addplot+[mark = square*, mark size = 1.5pt, mark repeat = 5, very thick, mark options={
    fill=green!60!black,
    draw=green!60!black
  }]
table[
    x=iter,
    y=l1_sflsmr,
    col sep=space
]{paper_01_sflsmr_residuals.dat};
\addlegendentry{sLSMR}

\nextgroupplot[
    title={Noise level: 1\%},
    ylabel={Relative error}
]
\addplot+[mark = *, mark size = 1.5pt, mark repeat = 5, very thick,
        blue!75!white]
table[
    x=iter,
    y=l1_lsqr,
    col sep=space
]{paper_01_sflsmr_errors.dat};
\addlegendentry{LSQR}

\addplot+[mark = triangle*, mark size = 1.5pt, mark repeat = 5, very thick,
        red!75!white]
table[
    x=iter,
    y=l1_sflsqr,
    col sep=space
]{paper_01_sflsmr_errors.dat};
\addlegendentry{sLSQR}

\addplot+[mark = square*, mark size = 1.5pt, mark repeat = 5, very thick, mark options={fill=green!60!black, draw=green!60!black}, green!60!black]
table[
    x=iter,
    y=l1_sflsmr,
    col sep=space
]{paper_01_sflsmr_errors.dat};
\addlegendentry{sLSMR}

\nextgroupplot[
    title={Noise level: 10\%},
    ylabel={Relative residual}
]
\addplot+[mark = *, mark size = 1.5pt, mark repeat = 5, very thick,
        blue!75!white]
table[
    x=iter,
    y=l2_lsqr,
    col sep=space
]{paper_01_sflsmr_residuals.dat};
\addlegendentry{LSQR}

\addplot+[mark = triangle*, mark size = 1.5pt, mark repeat = 5, very thick,
        red!75!white]
table[
    x=iter,
    y=l2_sflsqr,
    col sep=space
]{paper_01_sflsmr_residuals.dat};
\addlegendentry{sLSQR}

\addplot+[mark = square*, mark size = 1.5pt, mark repeat = 5, very thick, mark options={fill=green!60!black, draw=green!60!black}, green!60!black]
table[
    x=iter,
    y=l2_sflsmr,
    col sep=space
]{paper_01_sflsmr_residuals.dat};
\addlegendentry{sLSMR}

\nextgroupplot[
    title={Noise level: 10\%},
    ylabel={Relative error}
]
\addplot+[mark = *, mark size = 1.5pt, mark repeat = 5, very thick,
        blue!75!white]
table[
    x=iter,
    y=l2_lsqr,
    col sep=space
]{paper_01_sflsmr_errors.dat};
\addlegendentry{LSQR}

\addplot+[mark = triangle*, mark size = 1.5pt, mark repeat = 5, very thick,
        red!75!white]
table[
    x=iter,
    y=l2_sflsqr,
    col sep=space
]{paper_01_sflsmr_errors.dat};
\addlegendentry{sLSQR}

\addplot+[mark = square*, mark size = 1.5pt, mark repeat = 5, very thick, mark options={fill=green!60!black, draw=green!60!black}, green!60!black]
table[
    x=iter,
    y=l2_sflsmr,
    col sep=space
]{paper_01_sflsmr_errors.dat};
\addlegendentry{sLSMR}

\end{groupplot}

\end{tikzpicture}    
    \caption{Artificial test problem described in  Section~\ref{sec:residual-size},  
      with noise levels $1\%$ and $10\%$. Residual and error norms for the solution 
      recovered with LSQR, sLSQR, and sLSMR with Gaussian sketching.}
    \label{fig:artificial-01-sflsmr}
\end{figure}

\subsection{Bounds on the sketching error}\label{sec: bounds}
{In this section, we establish theoretical bounds for the sFLSQR and sFLSMR residuals.
Recall that, at each iteration, both methods construct the same subspace $\modZ_k$ and compute an approximate solution in $\mathrm{range}(\modZ_k)$ by solving two different sketched minimization problems.
In Theorem~\ref{thm:deterministic_bound}, we derive deterministic relations between the residuals produced by sFLSQR and sFLSMR and the best residual attainable in $\mathrm{range}(\modZ_k)$.}
In Corollary~\ref{thm:probabilistic_bound_LSQR} we further strengthen this result for sFLSQR and assuming Gaussian sketchings, obtaining a bound in expectation. 
For sFLSMR, deriving an analogous probabilistic bound appears considerably more challenging, as it would require detailed information about the spectrum of $\mathbf{A}$ and its interaction with the sketch. Despite this theoretical difficulty, our numerical experiments (Figure~\ref{fig:rho_0.8} and Section~\ref{sec:numexp}), demonstrate that sFLSMR with sketching consistently benefits from the decay of the singular values of $\bfA$, yielding a behavior that is even more favorable than in the sFLSQR case.



\begin{theorem}\label{thm:deterministic_bound}
    Let $\mxA\in \mathbb{R}^{m\times n}$, $\modZ_k\in \mathbb{R}^{n\times k}$, $\mxS\in \mathbb{R}^{s\times m}$ with $k\leq s \leq n$ and $\modZ_k$ generated via FGK \eqref{eq: FGK}. 
    Let us denote by $r_k^{\textrm{\emph{sFLSQR}}}$, $r_k^{\text{\emph{sFLSMR}}}$, and $r_k^{\textrm{\emph{opt}}}$ the sFLSQR, sFLSMR and optimal residual norms, respectively, where the latter is such that
    \begin{equation}\label{ropt}
r_k^{\textrm{\emph{opt}}}=\|\bfA\modZ_k\bfy_k^{\textrm{\emph{opt}}}-\bfb\|_2=\min_{\bfy\in\bbR^k}\|\bfA\modZ_k\bfy-\bfb\|_2.
    \end{equation}
    Then, if $\mxS \mxA \mxZ_k^{(p)}$ 
    and $\mxS \mxA^{\!\top} \mxA \mxZ_k^{(p)}$ are full rank,
    \begin{equation}\label{bound1}
r_k^{\textrm{\emph{sFLSQR}}}\leq r_k^{\textrm{\emph{opt}}} \sqrt{1 + \|(\mxS\mxQ)^\dagger \mxS \mxQ_\perp\|_2^2},
\end{equation}
\begin{equation}\label{bound2}
r_k^{\textrm{\emph{sFLSMR}}}\leq r_k^{\textrm{\emph{opt}}} \sqrt{1 + \|(\mxS\mxA^{\!\top}\mxQ)^\dagger \mxS\mxA^{\!\top} \mxQ_\perp\|_2^2},
\end{equation}
where $\mxQ$ is orthogonal with columns 
spanning $\mathrm{range}(\mxA\modZ_k)$ and $\mxQ_\perp$ is a basis for
its orthogonal complement.
\end{theorem}
\begin{proof}
For any $\bfy\in\mathbb{R}^{k}$, $\|\mxA\modZ_k\bfy-\bfb\|_2$ satisfies
\begin{align}
 \|\mxA\modZ_k\bfy-\bfb\|_2^2 &= \|\mxQ\mxQ^{\!\top}(\mxA\modZ_k\bfy-\bfb)\|_2^2+\|(\mxI-\mxQ\mxQ^{\!\top})(\mxA\modZ_k\bfy-\bfb)\|_2^2 \nonumber\\
&= \|\mxQ^{\!\top}(\mxA\modZ_k\bfy-\bfb)\|_2^2 + \|(\mxI-\mxQ\mxQ^{\!\top})\bfb\|_2^2.\label{res_sum}
\end{align}
Note that $r_k^{\textrm{{opt}}}$ in \eqref{ropt} is such that 
\begin{equation}\label{eq: FLSQRres}
r_k^{\textrm{{opt}}}=\mathrm{min}_{\bfy\in\bbR^k} \|\mxA\modZ_k\bfy-\bfb\|_2=\|(\mxI-\mxQ\mxQ^{\!\top})\bfb\|_2=\|\mxQ_\perp^{\!\top} \bfb\|_2,
\end{equation}
where the above residual is independent of the choice of $\bfy$. Therefore, when relating $r_k^{\textrm{{opt}}}$ to $r_k^{\text{sFLSQR}}$ and $r_k^{\text{sFLSMR}}$, we focus on bounding only on the first term in \eqref{res_sum}. 
Let ${\bfx}_k^{\text{sFLSQR}}=\modZ_k{\bfy}_k^{\text{sFLSQR}}$ be the sFLSQR solution defined in \eqref{eq:sFLSQR}. 
We have
\begin{align*}
\|\mxQ^{\!\top}(\mxA\modZ_k{\bfy}_k^{\text{sFLSQR}}-\bfb)\|_2 &= \|\mxQ^{\!\top}(\mxA\modZ_k(\mxS\mxA\modZ_k)^\dagger \mxS\bfb -\bfb)\|_2\\
&= \|\mxQ^{\!\top}(\mxQ(\mxS\mxQ)^\dagger \mxS \bfb -\bfb)\|_2 \\
&= \|(\mxS\mxQ)^\dagger \mxS \bfb - \mxQ^{\!\top}\bfb\|_2\\ 
& = \|((\mxS\mxQ)^\dagger \mxS - \mxQ^{\!\top})(\mxQ\mxQ^{\!\top} + \mxQ_\perp \mxQ_\perp^{\!\top}) \bfb\|_2  \\
&=\|((\mxS\mxQ)^\dagger \mxS \mxQ_\perp \mxQ_\perp^{\!\top}) \bfb\|_2\\
& \leq \|(\mxS\mxQ)^\dagger \mxS \mxQ_\perp\|_2\|\mxQ_\perp^{\!\top} \bfb\|_2,
\end{align*}
where, in deriving the last equality, we have used the fact that $(\bfS\bfQ)^\dagger(\bfS\bfQ)=\bfI$ (because $\bfS\bfQ$ has full column rank). Therefore
\begin{align*}
(r_k^{\text{sFLSQR}})^2&=\|\mxQ^{\!\top}(\mxA\modZ_k{\bfy}_k^{\text{sFLSQR}}-\bfb)\|_2^2+\|\mxQ_\perp^{\!\top} \bfb\|_2^2\\
&\leq \|(\mxS\mxQ)^\dagger \mxS \mxQ_\perp\|_2^2\|\mxQ_\perp^{\!\top} \bfb\|_2^2+\|\mxQ_\perp^{\!\top} \bfb\|_2^2,
\end{align*}
which leads to \eqref{bound1}. 
Let ${\bfx}_k^{\text{sFLSMR}}=\modZ_k{\bfy}_k^{\text{sFLSMR}}$ be the sFLSMR solution defined in \eqref{eq:sFLSMR}.
We have
\begin{align*}
 \|\mxQ^{\!\top}(\mxA\modZ_k{\bfy}_k^{\text{sFLSMR}}-\bfb)\|_2&= \|\mxQ^{\!\top}(\mxA\modZ_k(\mxS\mxA^{\!\top}\mxA\modZ_k)^\dagger \mxS\mxA^{\!\top}\bfb -\bfb)\|_2 \\
 &=\|\mxQ^{\!\top}(\mxQ(\mxS\mxA^{\!\top}\mxQ)^\dagger \mxS \mxA^{\!\top}\bfb -\bfb)\|_2 \\
 &=\|(\mxS\mxA^{\!\top}\mxQ)^\dagger \mxS\mxA^{\!\top} \bfb - \mxQ^{\!\top}\bfb\|_2\\
 & = \|((\mxS\mxA^{\!\top}\mxQ)^\dagger \mxS \mxA^{\!\top}- \mxQ^{\!\top})(\mxQ\mxQ^{\!\top} + \mxQ_\perp \mxQ_\perp^{\!\top}) \bfb\|_2  \\
 &=\|((\mxS\mxA^{\!\top}\mxQ)^\dagger \mxS \mxA^{\!\top} \mxQ_\perp \mxQ_\perp^{\!\top}) \bfb\|_2\\
 & \leq \|(\mxS\mxA^{\!\top}\mxQ)^\dagger \mxS \mxA^{\!\top} \mxQ_\perp \|_2 \|\mxQ_\perp^{\!\top} \bfb\|_2,
\end{align*}
where, in deriving the last equality, we have used the fact that $(\bfS\bfA^{\!\top}\bfQ)^\dagger(\bfS\bfA^{\!\top}\bfQ)=\bfI$ (because $\bfS\bfA^{\!\top}\bfQ$ has full column rank). Therefore
\begin{align*}
(r_k^{\text{sFLSMR}})^2&=\|\mxQ^{\!\top}(\mxA\modZ_k{\bfy}_k^{\text{sFLSMR}}-\bfb)\|_2^2+\|\mxQ_\perp^{\!\top} \bfb\|_2^2\\
&\leq \|(\mxS\mxA^{\!\top}\mxQ)^\dagger \mxS \mxA^{\!\top} \mxQ_\perp \|_2^2\|\mxQ_\perp^{\!\top} \bfb\|_2^2+\|\mxQ_\perp^{\!\top} \bfb\|_2^2,
\end{align*}
which leads to \eqref{bound2}.
\end{proof}
\begin{remark}\label{rem:spectral_decay_bounds}
We observe, at present only experimentally 
(see Figure~\ref{fig:various_rho}), that the 
bound \eqref{bound2} is smaller than 
\eqref{bound1} when the matrix $\bfA$ exhibits spectral decay.
This is also reflected in a smaller residual for s(F)LSMR
than for s(F)LSQR, although the bounds are not sufficiently
sharp to fully explain this behavior.

There is one particular case in which the difference in behavior
is especially clear. Suppose that $\bfA$ has rank $k$, which may be
viewed as the limiting case of a very sharp singular-value gap.
Then, after $k$ steps of s(F)LSQR/s(F)LSMR, the rank of $\mxA \modZ_k$
is also equal to $k$, provided that no breakdown has occurred.
In this case, we obtain
\[
r_k^{\text{sFLSMR}} = r_k^{\text{opt}}.
\]
Indeed, since
\[
\mathrm{range}(\mxQ) = \mathrm{range}(\mxA \modZ_k)
     = \mathrm{range}(\mxA),
\]
we have $\mxA^{\!\top} \mxQ_\perp = \bfzero$, and hence
\[
\|(\mxS\mxA^{\!\top} \mxQ)^\dagger
  \mxS\mxA^{\!\top} \mxQ_\perp\|_2^2 = 0.
\]
\end{remark}

Following established techniques in randomized numerical linear algebra, and specializing to Gaussian sketching matrices, we obtain the following bound in expectation for 
sFLSQR.
\begin{corollary}\label{thm:probabilistic_bound_LSQR}
Let $\mxA\in\mathbb{R}^{m\times n}$, $\modZ_k\in\mathbb{R}^{n\times k}$, $\mxS\in\mathbb{R}^{s\times m}$, with $k\leq s \leq n$,  $\modZ_k$ generated via FGK \eqref{eq: FGK} 
and $\bfS$ a Gaussian sketching matrix (i.e., each entry is an independent\ $\mathcal{N}(0,1)$ random variable). Then, the sketched FLSQR residual  norm $r_{k}^{\textrm{\emph{sFLSQR}}}$ satisfies
\begin{equation}
\mathbb{E}\left[(r_{k}^{\textrm{\emph{sFLSQR}}})^2\right]=
(r_k^{\textrm{\emph{opt}}})^2\cdot \left(1 + \frac{s}{s - k - 1}\right),
\end{equation}
where $r_k^{\textrm{\emph{opt}}}$ denotes the optimal residual norm for solutions in $\modZ_k$, defined as in \eqref{ropt}.
\end{corollary}
\begin{proof}
In the proof of Theorem~\ref{thm:deterministic_bound}, we showed that 
\[
(r_{k}^{\textrm{{sFLSQR}}})^2 =\|\mxA\modZ_k{\bfy}_{k}^{\text{sFLSQR}}-\bfb\|_2^2 = \|(\mxS\mxQ)^\dagger \mxS \mxQ_\perp \mxQ_\perp^{\!\top} \bfb\|_2^2 + \|\mxQ_\perp^{\!\top} \bfb\|_2^2,
\]
where $\mxQ = \mathrm{orth}(\mxA\mxZ)$ and $\mxQ_\perp$ denotes its orthogonal complement. 
Since $\mxS$ is Gaussian and the product between a Gaussian matrix and an orthonormal matrix is Gaussian, both $\mxS\mxQ$ and $\mxS\mxQ_\perp$ are Gaussians. Moreover, \( \mxS\mxQ\) is full rank with probability \(1\), so \((\mxS\mxQ)^\dagger\) is almost surely well-defined and, since $\mxQ$ and $\mxQ_\perp$ have orthogonal ranges, the matrices $\mxS\mxQ$ and $\mxS\mxQ_\perp$ are also independent. Then, by \cite[Proposition 10.1]{hmt} 
\[
\mathbb{E}\left[\|(\mxS\mxQ)^\dagger \mxS \mxQ_\perp \mxQ_\perp^{\!\top} \bfb\|_2^2\right] = \mathbb{E}\left[\|(\mxS \mxQ)^\dagger \|_F^2\right]\|\mxQ_\perp^{\!\top} \bfb\|_2^2,
\]
and, by \cite[Proposition 10.2]{hmt}, 
\[
\mathbb{E}\left[\|(\mxS \mxQ)^\dagger \|_F^2\right] = \frac{s}{s-k-1}.
\]
The claim follows by noting that, as shown in \eqref{eq: FLSQRres}, $\|\mxQ_\perp^{\!\top} \bfb\|_2 = r_{k}^\mathrm{opt}$. 
\end{proof}

\begin{remark} Although Theorem~\ref{thm:deterministic_bound} and Corollary~\ref{thm:probabilistic_bound_LSQR} are stated for approximate solutions to \eqref{eq:inverseproblem} 
belonging to the subspace $\modZ_k$ generated by $k$ steps of the FGK factorization \eqref{eq: FGK}, they extend to any comparison of residuals within the same approximation subspace.
\end{remark}

We conclude this section with a few illustrations that highlight the differences between the residuals attained by sFLSQR and sFLSMR 
and assess the sharpness of the bounds derived in Theorem~\ref{thm:deterministic_bound}. The experimental setup is the same as Section~\ref{par:setup}, i.e., we take $\bfA$ of size $1024\times 512$ with decaying singular values of the form \eqref{ill_decay_sv} (with $n=512$) and different values of $\rho$ (to explore different decay rates); the solution $\bfx_{\true}$ is chosen as the constant vector of ones and the problem is normalized so that $\|\bfb_{\true}\|_2=1$; the sketching is Gaussian, with $s = 2k+ 1$ rows, where $k$ is the maximum
number of iterations. We run the flexible solvers without truncation, so that the basis generated by all the solvers coincides with the LSQR one and, as done earlier in the same setting, we use the acronym sLSQR and sLSMR in place of sFLSQR and sFLSMR, respectively.

\begin{figure}
    \centering
    \begin{tikzpicture}

\begin{groupplot}[
    group style={
        group size=2 by 1,
        horizontal sep=1.6cm
    },
    ylabel style={yshift=-1em},
    width=6.7cm,
    height=5cm,
    xlabel={Iterations},
    ylabel={Relative residual},
    grid=both,
    tick label style={font=\small},
    label style={font=\small},
    legend style={
        font=\small,
        at={(1.0,1.05)},
        anchor=south east,
        legend columns=2
    }
]

\node[align=center] at ($(group c1r1.north)!0.5!(group c2r1.north) + (0,1.6cm)$)
{Noise level: 10\%, Decay: $\rho = 1.15$};

\nextgroupplot[
    title={}
]
\addplot+[mark = *, mark size = 1.5pt, mark repeat = 2, very thick,
        blue!75!white]
table[
    x=iter,
    y=r_flsqr,
    col sep=space
]{convergence_nl_10_rho_08.dat};
\addlegendentry{opt}

\addplot+[mark = triangle*, mark size = 1.5pt, mark repeat = 2, very thick,
        red!75!white]
table[
    x=iter,
    y=r_sflsqr,
    col sep=space
]{convergence_nl_10_rho_08.dat};
\addlegendentry{sLSQR}

\addplot+[mark = square*, mark size = 1.5pt, mark repeat = 2, very thick, mark options={fill=green!60!black, draw=green!60!black}, green!60!black]
table[
    x=iter,
    y=r_sflsmr,
    col sep=space
]{convergence_nl_10_rho_08.dat};
\addlegendentry{sLSMR}

\nextgroupplot[
    title={},
    ymode=log
]
\addplot+[mark = triangle*, mark size = 1.5pt, mark repeat = 2, very thick,
        red!75!white]
table[
    x=iter,
    y=r_sflsqr,
    col sep=space
]{convergence_nl_10_rho_08.dat};
\addlegendentry{sLSQR}

\addplot+[mark = square*, mark size = 1.5pt, mark repeat = 5, very thick, mark options={fill=green!60!black, draw=green!60!black}, green!60!black]
table[
    x=iter,
    y=r_sflsmr,
    col sep=space
]{convergence_nl_10_rho_08.dat};
\addlegendentry{sLSMR}

\addplot+[mark = triangle*, mark size = 1.5pt, mark repeat = 5, very thick, dashed, red!75!white]
table[
    x=iter,
    y=bound_sflsqr,
    col sep=space
]{convergence_nl_10_rho_08.dat};
\addlegendentry{\eqref{bound1}}

\addplot+[mark = square*, mark size = 1.5pt, mark repeat = 5, very thick, dashed, green!60!black, mark options={fill=green!60!black, draw=green!60!black}, green!60!black]
table[
    x=iter,
    y=bound_sflsmr,
    col sep=space
]{convergence_nl_10_rho_08.dat};
\addlegendentry{\eqref{bound2}}

\end{groupplot}

\end{tikzpicture}
    \caption{Artificial test problem described in Section~\ref{sec:residual-size}, with singular values decay as in \eqref{ill_decay_sv}. Left frame: optimal, sLSQR, and sLSMR  residual norms (in the same approximation subspace) versus iteration count. Right frame: sLSQR, and sLSMR  residual norms and their bounds given in Theorem ~\ref{thm:deterministic_bound} versus iteration count.}
    \label{fig:rho_0.8}
\end{figure}
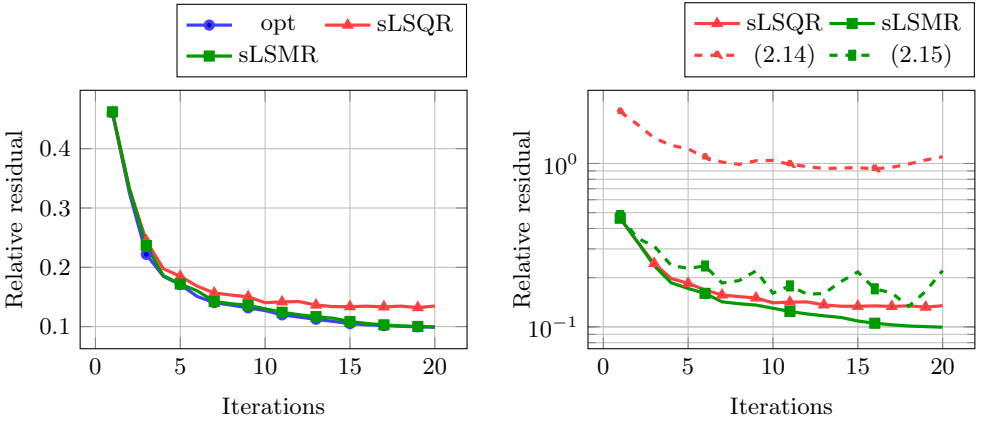

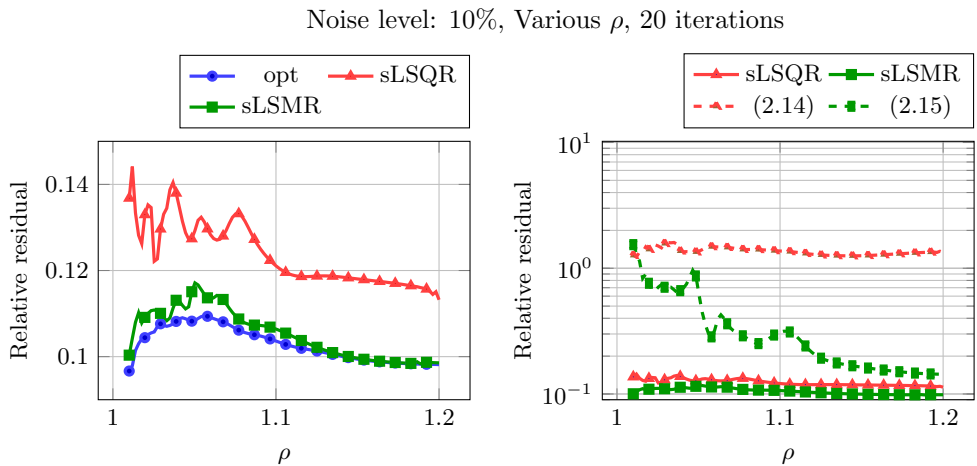
\begin{figure}
    \centering
\begin{tikzpicture}

\begin{groupplot}[
    group style={
        group size=2 by 1,
        horizontal sep=1.75cm
    },
    ylabel style={yshift=-0.5em},
    width=6.5cm,
    height=5cm,
    xlabel={$\rho$},
    ylabel={Relative residual},
    grid=both,
    tick label style={font=\small},
    label style={font=\small},
    title style={font=\small},
    legend style={
        font=\small,
        at={(1.0,1.05)},
        anchor=south east,
        legend columns=2
    }
]

\node[align=center] at ($(group c1r1.north)!0.5!(group c2r1.north) + (0,1.6cm)$)
{Noise level: 10\%, Various $\rho$, $20$ iterations};

\nextgroupplot[
    ymin=0.09,
    ymax=0.15
]
\addplot+[mark = *, mark size = 1.5pt, mark repeat = 5, very thick,
        blue!75!white]
table[
    x=rho,
    y=r_rho_flsqr,
    col sep=space
]{convergence_nl_10_various_rho.dat};
\addlegendentry{opt}

\addplot+[mark = triangle*, mark size = 1.5pt, mark repeat = 5, very thick,
        red!75!white]
table[
    x=rho,
    y=r_rho_sflsqr,
    col sep=space
]{convergence_nl_10_various_rho.dat};
\addlegendentry{sLSQR}

\addplot+[mark = square*, mark size = 1.5pt, mark repeat = 5, very thick, mark options={fill=green!60!black, draw=green!60!black}, green!60!black]
table[
    x=rho,
    y=r_rho_sflsmr,
    col sep=space
]{convergence_nl_10_various_rho.dat};
\addlegendentry{sLSMR}

\nextgroupplot[
    ymode=log,
    ymin=0.09,
    ymax=10.22
]
\addplot+[mark = triangle*, mark size = 1.5pt, mark repeat = 5, very thick,
        red!75!white]
table[
    x=rho,
    y=r_rho_sflsqr,
    col sep=space
]{convergence_nl_10_various_rho.dat};
\addlegendentry{sLSQR}

\addplot+[mark = square*, mark size = 1.5pt, mark repeat = 5, very thick,
        mark options={fill=green!60!black, draw=green!60!black}, green!60!black]
table[
    x=rho,
    y=r_rho_sflsmr,
    col sep=space,
    mark options={fill=green!60!black, draw=green!60!black}, green!60!black
]{convergence_nl_10_various_rho.dat};
\addlegendentry{sLSMR}

\addplot+[mark = triangle*, mark size = 1.5pt, mark repeat = 5, very thick, dashed, red!75!white]
table[
    x=rho,
    y=bound_rho_sflsqr,
    col sep=space
]{convergence_nl_10_various_rho.dat};
\addlegendentry{\eqref{bound1}}

\addplot+[mark = square*, mark size = 1.5pt, mark repeat = 5, very thick, dashed, mark options={fill=green!60!black, draw=green!60!black}, green!60!black]
table[
    x=rho,
    y=bound_rho_sflsmr,
    col sep=space
]{convergence_nl_10_various_rho.dat};
\addlegendentry{\eqref{bound2}}

\end{groupplot}

\end{tikzpicture}    
    \caption{Class of artificial test problems obtained from the one described in  Section~\ref{sec:residual-size}, with singular values decay as in \eqref{ill_decay_sv} and different $\rho$ values. Left frame:  optimal, 
    sLSQR, and sLSMR  residual norms in the same approximation subspace of dimension 20, 
    versus $\rho$. Right frame:  sLSQR, and sLSMR  residual norms at the 20th iteration and their bounds given in Theorem ~\ref{thm:deterministic_bound}, versus $\rho$.}
    \label{fig:various_rho}
\end{figure}

In Figure~\ref{fig:rho_0.8} we fix the singular value decay rate and run 
sLSQR and sLSMR.
On the left, we plot the optimal residual norm (opt) attainable 
within the considered approximation subspace $\bfZ_k^{(p)}$ (coinciding, under our assumptions, with the LSQR residual norm), 
and the one obtained by sLSQR and 
sLSMR, varying $k$. This plot is similar to the 
one already shown in Figure~\ref{fig:artificial-01-sflsmr}. 
On the right, 
we plot the residuals of the two methods 
and compare them with the 
upper bounds from Theorem~\ref{thm:deterministic_bound}. Apart from observing again that, as the iteration evolve, the sLSMR residual norm follows the optimal one more closely than sLSQR, we can also appreciate that the bound \eqref{bound2} consistently follows the sLSMR residual along the iterations; the same is not true for the bound \eqref{bound1} and the sLSQR residual. In Figure~\ref{fig:various_rho}, instead of evaluating the residual at each iteration, we perform 20 iterations (fixed for every solver) and vary the value of $\rho$ from 1.01 to 1.15. Looking at both frames, we can positively state once more that the sLSMR residual norm more closely follows the optimal one, and its bounds \eqref{bound2} are tighter, with respect to their sLSQR counterparts and across the whole range of $\rho$ values. Heuristically, the different behavior of the bounds in Theorem~\ref{thm:deterministic_bound} can be motivated by noting that $\bfA$ appears explicitly only in the sFLSMR bound \eqref{bound2}, implying a tighter bound when the singular value decay is quicker (larger $\rho$'s); see also Remark~\ref{rem:spectral_decay_bounds}. Moreover, we generally observe a decrease in all the residual norm values as $\rho$ increases shortly after $\rho=1.05$. This is because, as already commented in the previous sections, the relative LSQR residual norm tends to stabilize around the noise level, which happens within 20 iterations only for problems with a quicker singular value decay. 

\subsection{Sketch selection}
\label{sec:sketching}

The theory of randomized oblivious embeddings 
enables to choose a random matrix $\bfS \in \mathbb R^{s \times n}$ such that, 
for vectors $\bfv$ belonging to a certain subspace $\mathcal V$, the 
$\epsilon$-embedding property
$(1 - \epsilon) \| \bfv \| \leq \| \bfS\bfv \| \leq (1 + \epsilon) \| \bfv \|_2$
holds with high probability. 
The literature offers several options for sketching operators, but the specific choice remains largely irrelevant to the results in this paper. We recall that Gaussian sketches offer the strongest theoretical guarantees but suffer from high computational overhead due to their unstructured nature. Other common choices, such as subsampled trigonometric transforms or sparse arrays, provide better computational performance but weaker theoretical bounds. We used Gaussian sketching for the artificial examples in Sections~\ref{sec:residual-size} and \ref{sec:sflsmr}, but we will use Countsketch for all subsequent numerical experiments. We refer the interested reader to \cite{martinsson2020randomized} for a
recent overview on this topic.

\section{Numerical experiments}
\label{sec:numexp}

This section is devoted to demonstrate the effectiveness
of sFLSQR and sFLSMR in the context of imaging inverse 
problems. 
The ``flexibility'' in the Krylov method is used either to impose a 
low-rank structure in the solution by truncating the basis vectors 
at each step or to perform matvec products with an unmatched transpose. In all the experiments,  incomplete orthogonalization is limited to the last $\ell=2$ vectors.

Our codes are available 
online at 
\texttt{https://github.com/robol/sFLSQR}; the experiments are
provided in the form of Julia notebooks
running on Julia 1.11.7. 
All tests
have been run on a 
system with an 
AMD Ryzen 7 3700X 8-Core Processor, 
32 GB of RAM, and 
a NVIDIA GeForce 1030 GT GPU
(the latter 
is used only 
in the CT scan examples
of Section~\ref{sec:numexp-astra}). 

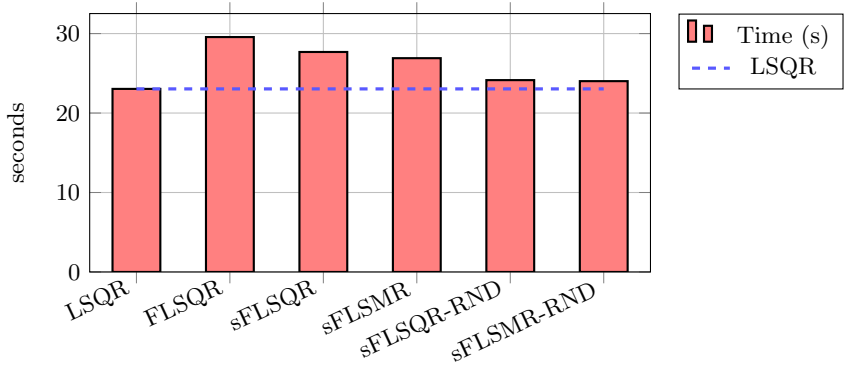
\begin{figure}
    \centering
    \begin{tikzpicture}

\begin{axis}[
    width=9cm,
    height=5cm,
    ybar,
    bar width=18pt,
    ylabel={seconds},
    grid=both,
    ymin=0,
    symbolic x coords={
        LSQR,
        FLSQR,
        sFLSQR,
        sFLSMR,
        sFLSQR-RND,
        sFLSMR-RND
    },
    xtick=data,
    x tick label style={
        rotate=25,
        anchor=east,
        font=\small
    },
    tick label style={font=\small},
    label style={font=\small},
    ylabel style={yshift=-0.25cm},
    legend style={
        font=\small,
        at={(1.05,1.0)},
        anchor=north west,
        rotate = 0
    }
]

\addplot+[
    ybar,
    fill,
    draw=black,
    thick,
    fill=red!50!white
]
table[
    x=method,
    y=time,
    col sep=space
]{paper_02_times.dat};
\addlegendentry{Time (s)}

\addplot+[
    sharp plot,
    mark=none,
    very thick,
    dashed,
    blue!65!white,
    legend image code/.code={
        \draw[blue!65!white, very thick, dashed] (0cm,0cm) -- (0.5cm,0cm);
    }
]
table[
    x=method,
    y=lsqr_baseline,
    col sep=space
]{paper_02_times.dat};
\addlegendentry{LSQR}

\end{axis}

\end{tikzpicture}
    \caption{Image deblurring and inpainting test problem. Timings for running $50$ iterations of LSQR, and its flexible and sketched counterparts. The flexible variants use low-rank truncation with a fixed rank of $30$, and we rely on 
    CountSketch for the sketching. The methods with label ending in
    RND use a randomized SVD instead of the usual truncated SVD to perform low-rank truncation. }
    \label{fig:cameraman-timings}
\end{figure}

\begin{figure}
    \centering
    \includegraphics[width=0.85\linewidth]{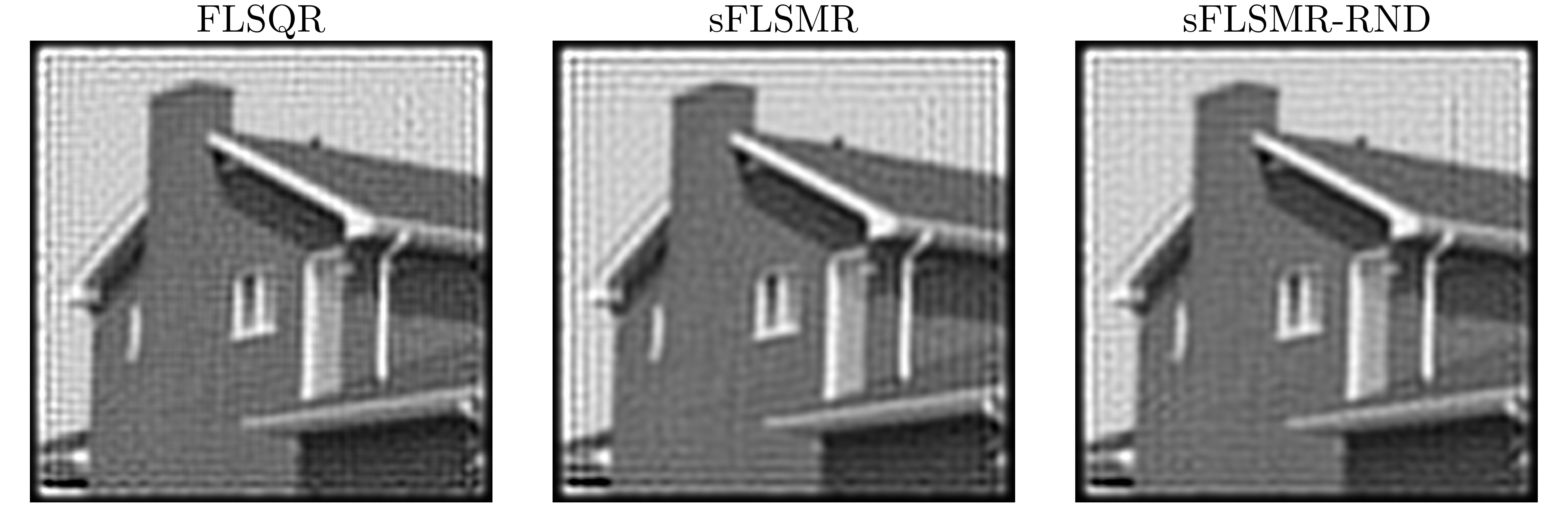}
    \caption{Image deblurring and inpainting test problem. Images recovered, with FLSQR and sFLSMR using the standard and 
      randomized SVD, at the end of the iterations.}
    \label{fig:cameraman-rnd}
\end{figure}

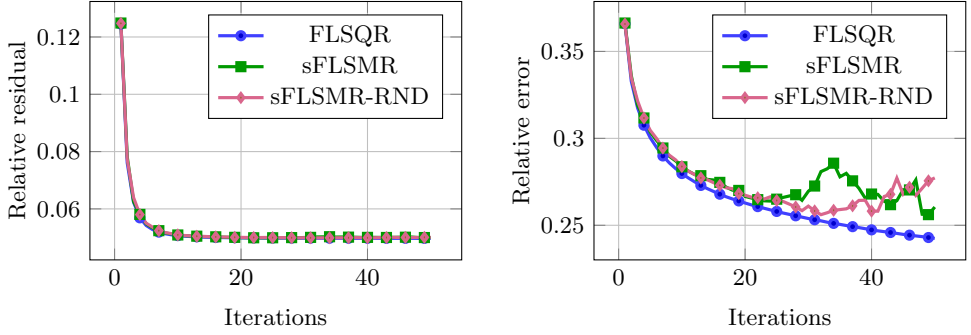
\begin{figure}
    \centering
\begin{tikzpicture}

\begin{groupplot}[
    group style={
        group size=2 by 1,
        horizontal sep=1.75cm
    },
    ylabel style={yshift=-0.75em},
    width=6.5cm,
    height=5cm,
    xlabel={Iterations},
    grid=both,
    tick label style={font=\small},
    label style={font=\small},
    title style={font=\small},
    yticklabel style={
        /pgf/number format/fixed,
        /pgf/number format/precision=2,
    },    
    legend style={
        font=\small,
        at={(0.95,0.95)},
        anchor=north east,
        legend columns=1
    }
]

\nextgroupplot[
    ylabel={Relative residual}
]
\addplot+[mark = *, mark size = 1.5pt, mark repeat = 3, very thick,
    blue!75!white]
table[
    x=iter,
    y=flsqr,
    col sep=space
]{cameraman_rnd_residuals.dat};
\addlegendentry{FLSQR}

\addplot+[mark = square*, mark size = 1.5pt, mark repeat = 3, very thick,
    mark options={fill=green!60!black, draw=green!60!black}, green!60!black]
table[
    x=iter,
    y=sflsmr,
    col sep=space
]{cameraman_rnd_residuals.dat};
\addlegendentry{sFLSMR}

\addplot+[mark = diamond*, mark size = 1.5pt, mark repeat = 3, very thick, 
    purple!60!white]
table[
    x=iter,
    y=sflsmr_rnd,
    col sep=space
]{cameraman_rnd_residuals.dat};
\addlegendentry{sFLSMR-RND}

\nextgroupplot[
    ylabel={Relative error}
]
\addplot+[mark = *, mark size = 1.5pt, mark repeat = 3, very thick, blue!75!white]
table[
    x=iter,
    y=flsqr,
    col sep=space
]{cameraman_rnd_errors.dat};
\addlegendentry{FLSQR}

\addplot+[mark = square*, mark size = 1.5pt, mark repeat = 3, very thick, mark options={fill=green!60!black, draw=green!60!black}, green!60!black]
table[
    x=iter,
    y=sflsmr,
    col sep=space
]{cameraman_rnd_errors.dat};
\addlegendentry{sFLSMR}

\addplot+[mark = diamond*, mark size = 1.5pt, mark repeat = 3, very thick, purple!60!white]
table[
    x=iter,
    y=sflsmr_rnd,
    col sep=space
]{cameraman_rnd_errors.dat};
\addlegendentry{sFLSMR-RND}

\end{groupplot}

\end{tikzpicture}    
    \caption{Image deblurring and inpainting test problem. Convergence history for residuals and errors computed by FLSQR and sFLSMR using the standard and 
      randomized SVD.}
    \label{fig:cameraman-convergence}
\end{figure}

\subsection{Image deblurring and inpainting}
As a first test case, we consider the  \texttt{house} test image of size $512\times 512$ pixels, with
the same blurring and subsampling considered in Section~\ref{sect:intro}. 
%
We test the computational complexity, and we compare the 
runtime, of the following methods, for a fixed number of 
$50$ iterations:
\begin{description}
    \item \textit{LSQR} The LSQR implementation
      from the \texttt{IterativeSolvers.jl} Julia package\footnote{We have also tested the timings with our own implementation of LSQR, to make sure that the comparison was as fair as possible, and there were no 
      appreciable differences.}.
    \item \textit{FLSQR} Our own flexible LSQR implementation, using 
      low-rank truncation \linebreak[4]$\tau_r(\bfv)$ \eqref{def:lowr-trc} applied to the basis vectors, with truncation rank $r=30$; this is the approach suggested in \cite{gazzola-lowrank}.
    \item \textit{sFLSQR} The sketched version of flexible LSQR described 
      in Section~\ref{sec:sflsqr}, using 
      CountSketch as sketching. The sketching size is 
      equal to 101 (i.e., twice the maximum number of iterations plus one).
    \item \textit{sFLSMR} The sketched version of flexible LSMR described in Section~\ref{sec:sflsmr}, with the same setup of sFLSQR.
    \item \textit{sFLSQR-RND} The sketched version of flexible LSQR, using the 
      Randomized SVD 
      by \cite{hmt} instead of a full truncated 
      SVD to perform low-rank truncation of the basis vectors.
    \item \textit{sFLSMR-RND} The same as above, but with flexible LSMR.
\end{description}

The timings are reported in the bar plot of Fig.~\ref{fig:cameraman-timings}. As expected, LSQR delivers good performances, and FLSQR is 
slower, mainly due to the full reorthogonalization. The sketched 
methods are significantly more competitive, and require a similar 
computational effort. When exploiting the randomized SVD, the
computational time is essentially the same as that of the standard LSQR.

We next compare the performance in terms of (quality of the) reconstructions for FLSQR, sFLSQR and sFLSQR-RND. 
As visible from Fig.~\ref{fig:cameraman-rnd} (reporting reconstructions) and Fig.~\ref{fig:cameraman-convergence} (reporting relative residuals and errors histories), the results
obtained with the standard and randomized SVD are very close. 
The sketching and randomization in the SVD only come at a small loss 
in accuracy at the end of the convergence history. 

\subsection{Computed Tomography (CT)}
\label{sec:numexp-astra}
As mentioned in Section \ref{sect:intro}, when dealing with large-scale CT 
problems, an efficient backprojection (i.e., multiplication by the transpose of the forward operator) may only be available 
approximately due to the organization of data structures in the GPU; see \cite{astra}.
 Typically, commonly used algebraic iterative methods (such as Landweber) may not converge in this situation, although one can introduce nontrivial modifications to ensure convergence to the solution of a slightly perturbed problem; see \cite{CT2}. Quite recently, the authors of \cite{CT1} propose to handle $\bfA^\sharp \approx \bfA^{\!\top}$ 
with Krylov methods that do not rely on $\bfA^{\!\top}$, such as AB-GMRES and BA-GMRES; see also \cite{ABBAGMRES}. 

Here we consider solving CT problems with unmatched projector/backprojector pairs via FLSQR and FLSMR, which are rooted in the FGK factorization \eqref{eq:flsqr-relations}, and therefore extend LSQR and LSMR to situations where $\bfA^\sharp\bfA$ and $\bfA\bfA^\sharp$ are not symmetric by performing full orthogonalization of the basis vectors. If we consider no (further) modifications to the solution space basis vectors (i.e., we take $\tau(\bfv)=\bfv$ in \eqref{eq: modbasis}), then both FLSQR and FLSMR build the Krylov subspace $\calK_k(\bfA^\sharp \bfA, \bfA^\sharp \bfb)$ for the approximation of a solution and are mathematically equivalent to AB-GMRES and BA-GMRES, respectively. Still assuming $\tau(\bfv)=\bfv$, the sketched version of these solvers, i.e., sFLSQR and sFLSMR, are equivalent to sketched AB-GMRES and sketched BA-GMRES, respectively.  
In the following we experimentally show that switching to sFLSQR and sFLSMR allows to work with a non-orthogonal basis for $\calK_k(\bfA^\sharp \bfA, \bfA^\sharp \bfb)$ and still solve associated minimization problems accurately and at a low cost. Even if not investigated here, this holds also in cases where, in addition to introducing flexibility to handle $\bfA^\sharp\approx\bfA^{\!\top}$, one considers $\tau(\bfv)\neq \bfv$ to enforce additional regularity into the solution.

The following tests employ the 
\emph{ASTRA Toolbox}, an open source 
package implementing high-performance GPU primitives for 2D and 3D tomography \cite{astra}, to generate unmatched transposes $\bfA^\sharp$.


\paragraph{2D CT problem} 
 We consider the 2D $1024 \times 1024$ 
phantom generated with 
the \texttt{shepp\_logan} function of the \emph{ASTRA Toolbox}, and simulate a 
CT acquisition with parallel geometry using $5600$ rays at $180$ equispaced angles between 0 and 180 degrees. This produces a least squares problem \eqref{eq:inverseproblem} with coefficient matrix of size 
$1008000 \times 1048576$, which is slightly underdetermined. 
The right hand side vector is polluted with Gaussian noise of level $5\%$.

We emphasize again that, when using the GPU operations in the \emph{ASTRA Toolbox}, the backprojection
is only an approximation of the transpose operator, so performing 
LSQR is not directly possible, and one has to resort to  
FLSQR even if no basis vector truncation is involved. We compare the results of running FLSQR, and our versions
of sFLSQR and sFLSMR. In our tests we have verified that, given random vectors $\bfx$ and $\bfy$ of unit 2-norm, 
we have
\[
| \bfx^{\!\top} \bfA \bfy - \bfy^{\!\top} \bfA^\sharp \bfx| \approx 4 \cdot 10^{-2}\,.
\]

The recovered solutions at iteration $15$ for both FLSQR and sFLSMR are displayed in Fig.~\ref{fig:astra-2d-reconstruction}. The residual and error plots are reported in Fig.~\ref{fig:astra-2d-convergence}. These results demonstrate that the presented methods are competitive with FLSQR$\!$;  
in particular, sFLSMR is reliable and robust, and delivers
residuals in line with the deterministic methods, as well as the best error
among all techniques. 

\begin{figure}
    \centering
    \includegraphics[width=0.9\linewidth]{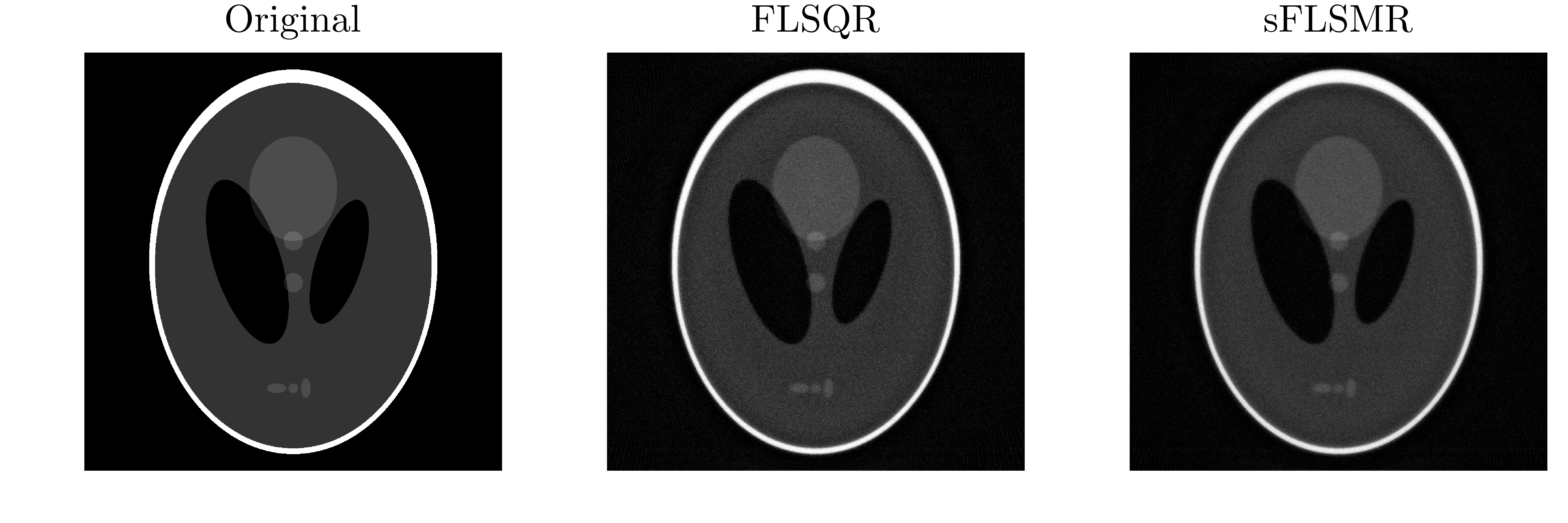}
    \caption{2D CT test problem. Original phantom; reconstructed phantoms at the 15th iteration of both the FLSQR and the sFLSMR solvers.} 
    \label{fig:astra-2d-reconstruction}
\end{figure}

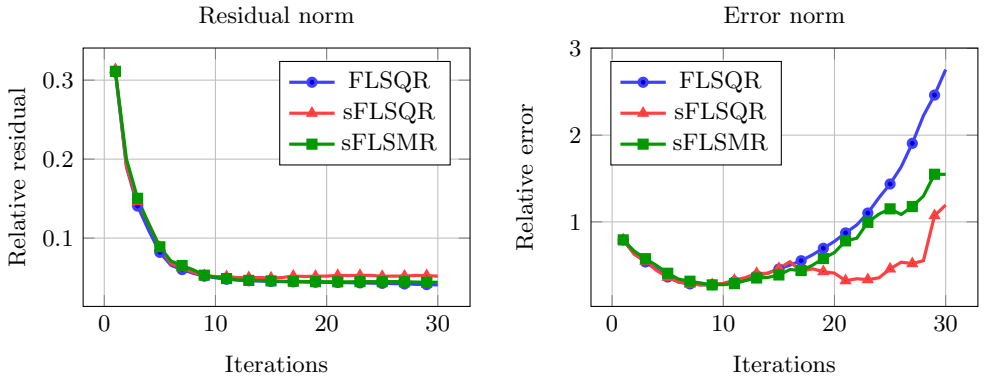
\begin{figure}
    \centering
\begin{tikzpicture}

\begin{groupplot}[
    group style={
        group size=2 by 1,
        horizontal sep=1.6cm
    },
    ylabel style={yshift=-1em},
    width=6.7cm,
    height=5cm,
    xlabel={Iterations},
    grid=both,
    tick label style={font=\small},
    label style={font=\small},
    title style={font=\small},
    legend style={
        font=\small,
        at={(0.95,0.95)},
        anchor=north east,
        legend columns=1
    }
]

\nextgroupplot[
    title={Residual norm},
    ylabel={Relative residual}
]
\addplot+[mark=none, very thick,
    blue!75!white, mark = *, mark size = 1.5pt, mark repeat = 2]
table[
    x=iter,
    y=flsqr,
    col sep=space
]{astra_2D_residuals.dat};
\addlegendentry{FLSQR}

\addplot+[mark=none, very thick,
    red!75!white, mark = triangle*, mark size = 1.5pt, mark repeat = 2]
table[
    x=iter,
    y=sflsqr,
    col sep=space
]{astra_2D_residuals.dat};
\addlegendentry{sFLSQR}

\addplot+[mark=none, very thick, mark options={fill=green!60!black, draw=green!60!black}, green!60!black, mark = square*, mark size = 1.5pt, mark repeat = 2]
table[
    x=iter,
    y=sflsmr,
    col sep=space
]{astra_2D_residuals.dat};
\addlegendentry{sFLSMR}

\nextgroupplot[
    title={Error norm},
    ylabel={Relative error},
    legend style={
        font=\small,
        at={(0.05,0.95)},
        anchor=north west,
        legend columns=1
    }
]
\addplot+[mark size = 1.5pt, very thick,
    blue!75!white, mark = *, mark repeat = 2]
table[
    x=iter,
    y=flsqr,
    col sep=space
]{astra_2D_errors.dat};
\addlegendentry{FLSQR}

\addplot+[mark size = 1.5pt, mark = triangle*, 
    very thick,
    red!75!white, mark repeat = 2]
table[
    x=iter,
    y=sflsqr,
    col sep=space
]{astra_2D_errors.dat};
\addlegendentry{sFLSQR}

\addplot+[mark size = 1.5pt, very thick, mark options={fill=green!60!black, draw=green!60!black}, green!60!black, mark = square*, mark repeat = 2]
table[
    x=iter,
    y=sflsmr,
    col sep=space
]{astra_2D_errors.dat};
\addlegendentry{sFLSMR}

\end{groupplot}

\end{tikzpicture}    
    \caption{2D CT test problem. Residual and error plots versus iteration count. 
    }
    \label{fig:astra-2d-convergence}
\end{figure}

The running times for $30$ iterations of 
FLSQR is of $6.78$ seconds, while 
sFLSQR and sFLMR both require around $3.25$ seconds. 
Running LSQR ignoring the fact that the transpose is approximated
requires $1.7$ seconds. We remark that, for this test case, 
the approximation in the transpose is accurate enough that 
running unmodified LSQR is indeed a viable option, and the 
results obtained in this way are typically good reconstructions. 
However, the fact that sketched flexible methods can be competitive
opens the door to performing even more aggressive approximation 
for the backprojection operator. 

\paragraph{3D CT problem} 
We repeat a similar test in the 3D case, using a 3D version of the Shepp-Logan phantom of size $256 \times 256 \times 256$, and 
using the GPU implementation of the transpose, both available within the \emph{ASTRA Toolbox}. During 
our tests, we have verified that, given random vectors $\bfx$ and $\bfy$ of unit 2-norm, 
\[
\text{$| \bfx^{\!\top} \bfA \bfy - \bfy^{\!\top} \bfA^\sharp \bfx|$ is between $10^{-3}$
and $10^{-2}$.}
\]
Therefore the approximate transpose is quite accurate, 
but not as the machine precision of $\approx 10^{-7}$ would 
require (this example is in single precision). The 3D CT parallel rays scanning geometry is defined by a  
mesh of $100 \times 100$
detectors and $120$ equispaced angles between $0$ and $180$ degrees. Hence, the coefficient matrix for the least square problem in \eqref{eq:inverseproblem} has size $1200000 \times 16777216$ and is fairly undetermined.
The noise level in this test is $10\%$.

\begin{figure}
    \centering
    \includegraphics[width=0.9\linewidth]{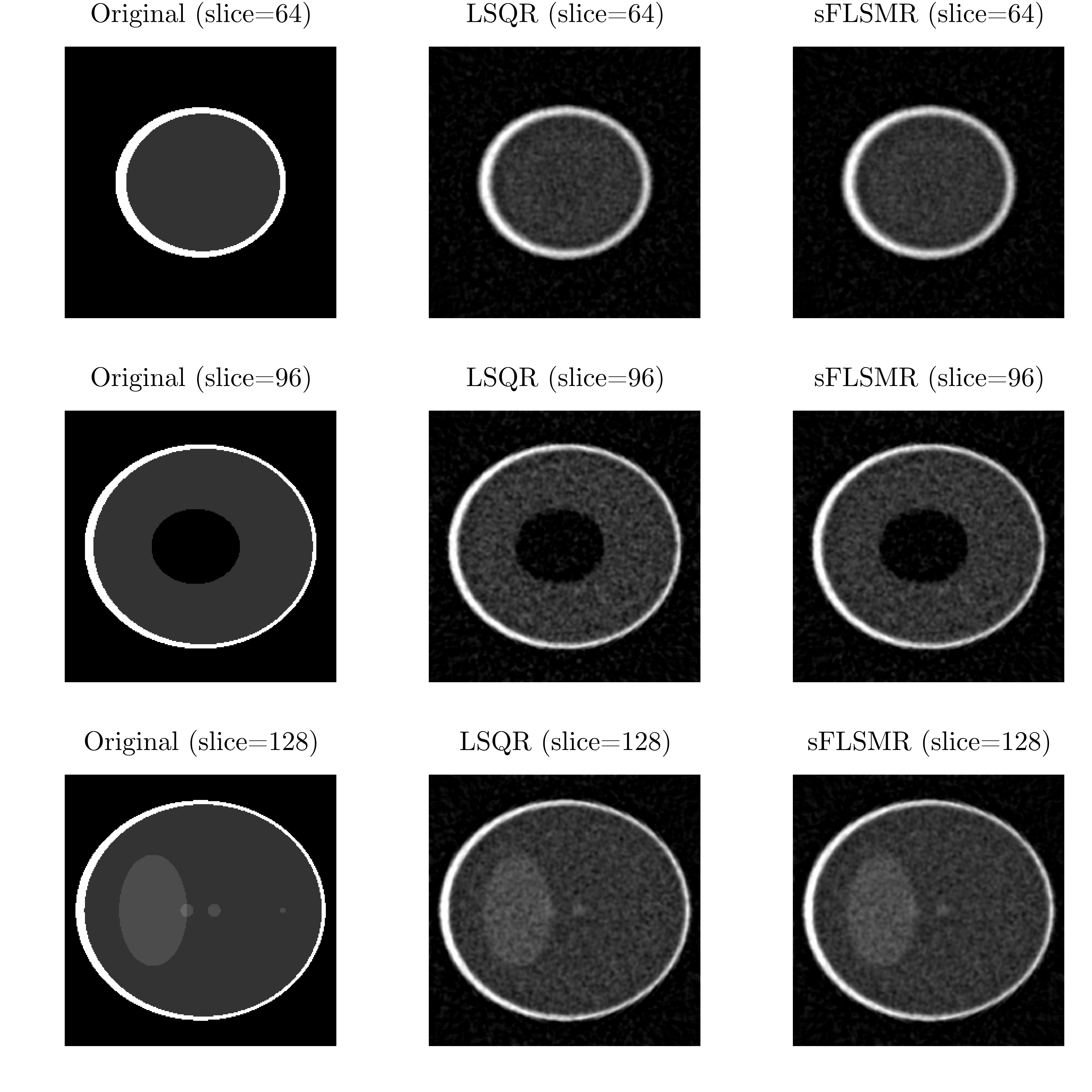}
    \caption{3D CT test problem. $3$ slices (with index $64$, $96$, and $128$, out of $256$) of the original and reconstructed 3D Shepp-Logan phantom; the reconstructions are obtained via 
    $12$ iterations of LSQR and sFLSMR. 
    }
    \label{fig:astra-3d-reconstruction}
\end{figure}

\begin{figure}
    \centering
    \begin{tikzpicture}
    
    \begin{groupplot}[
        group style={
            group size=2 by 1,
            horizontal sep=1.6cm
        },
        ylabel style={yshift=-1em},
        width=6.7cm,
        height=5cm,
        xlabel={Iterations},
        grid=both,
        tick label style={font=\small},
        label style={font=\small},
        title style={font=\small},
        legend style={
            font=\small,
            at={(0.95,0.95)},
            anchor=north east,
            legend columns=1
        }
    ]
    
    \nextgroupplot[
        title={Residual norm},
        ylabel={Relative residual}
    ]
    \addplot+[mark = *, mark size = 1.5pt, mark repeat = 2, very thick,
    blue!75!white]
    table[
        x=iter,
        y=lsqr,
        col sep=space
    ]{astra_3D_residuals.dat};
    \addlegendentry{LSQR}
    
    \addplot+[mark = triangle*, mark size = 1.5pt, mark repeat = 2, very thick,
    red!75!white]
    table[
        x=iter,
        y=sflsqr,
        col sep=space
    ]{astra_3D_residuals.dat};
    \addlegendentry{sFLSQR}
    
    \addplot+[mark = square*, mark size = 1.5pt, mark repeat = 2, very thick, mark options={fill=green!60!black, draw=green!60!black}, green!60!black]
    table[
        x=iter,
        y=sflsmr,
        col sep=space
    ]{astra_3D_residuals.dat};
    \addlegendentry{sFLSMR}
    
    \nextgroupplot[
        title={Error norm},
        ylabel={Relative error}
    ]
    \addplot+[mark = *, mark size = 1.5pt, mark repeat = 5, very thick,
    blue!75!white]
    table[
        x=iter,
        y=lsqr,
        col sep=space
    ]{astra_3D_errors.dat};
    \addlegendentry{LSQR}
    
    \addplot+[mark = triangle*, mark size = 1.5pt, mark repeat = 2, very thick,
    red!75!white]
    table[
        x=iter,
        y=sflsqr,
        col sep=space
    ]{astra_3D_errors.dat};
    \addlegendentry{sFLSQR}
    
    \addplot+[mark = square*, mark size = 1.5pt, mark repeat = 2, very thick, mark options={fill=green!60!black, draw=green!60!black}, green!60!black]
    table[
        x=iter,
        y=sflsmr,
        col sep=space
    ]{astra_3D_errors.dat};
    \addlegendentry{sFLSMR}
    
    \end{groupplot}
    
    \end{tikzpicture}    
    \caption{3D CT test problem. Residual and error plots versus iteration count.}
    \label{fig:astra-3d-convergence}
\end{figure}
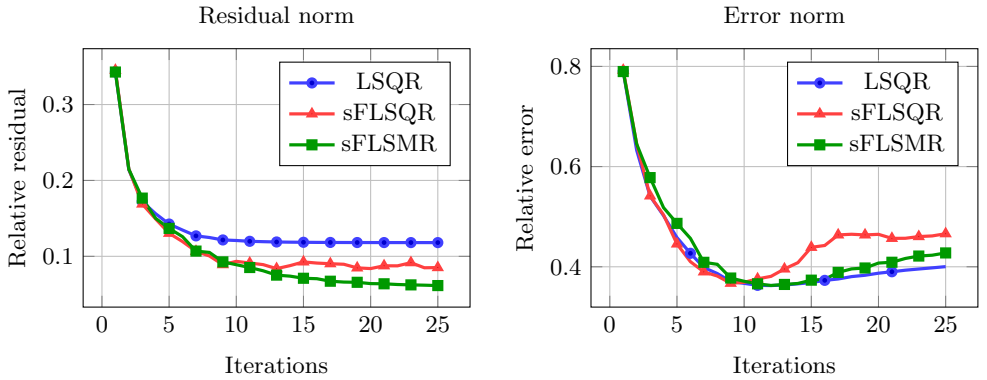

The recovered solutions (selected slices of the considered 3D phantom) for both LSQR and sFLSMR are displayed in Fig.~\ref{fig:astra-3d-reconstruction}. The residual and error plots for LSQR, sFLSQR and sFLSMR are reported in Fig.~\ref{fig:astra-3d-convergence}. These results are somewhat similar to the ones obtained for the 2D case, and demonstrate that the presented methods are competitive with LSQR.  
In particular, sFLSMR is reliable and robust: it delivers the lowest relative error and residuals, allowing the latter to decrease, while the LSQR and sFLSQR ones somewhat stagnate. Although not much improvements in the reconstruction quality are visible looking at the phantom slices in Fig.~\ref{fig:astra-3d-reconstruction}, sketching in the unmatched-transpose-aware sFLSMR method seems to positively affect the quality of the residual approximation (recall that, since $\bfA^\sharp\neq\bfA^{\!\top}$, LSQR does not return the exact residual norm). This may impact the quality of the computed reconstructions if the iterative solvers are automatically stopped by, say, the discrepancy principle. Indeed, looking at Fig.~\ref{fig:astra-3d-convergence}, we can clearly see that the LSQR relative residual norms seem to stabilize well above the noise level (10\%), even when the approximate solution is already starting to converge to an unregularized solution, leading to under-regularization; on the contrary, the sFLSMR residual hits the noise level at around iteration 9 or 10, i.e., just before the minimum relative error is computed and therefore before encountering semiconvergence, leading to slight over-regularization but a better-quality reconstruction than LSQR.

\section{Conclusions and outlook}
We have proposed two randomized flexible \linebreak[4]Krylov solvers, sFLSQR and sFLSMR, for large-scale least-squares problems. These methods combine the flexibility of the FGK-based solvers FLSQR and FLSMR with randomized sketching techniques, in order to alleviate the computational burden associated with the long recurrence relations required by flexible Krylov methods.

The main idea underlying sFLSQR and sFLSMR is to perform only partial orthogonalization when generating the basis vectors and to compute the solution updates by solving sketched and projected least-squares problems. This strategy significantly reduces the cost of orthogonalization and storage (thereby reducing computational time), while preserving the ability of flexible Krylov methods to incorporate structural information into the approximation subspace or to handle inexact applications of the transpose operator, which is particularly relevant when solving linear inverse problems. This is confirmed by numerical experiments on imaging problems. Additional randomization techniques, such as randomized low-rank approximations used in the basis truncation step, can be naturally accommodated, further reducing the computational cost. 
Our theoretical analysis establishes bounds that relate the residual norms obtained by the sketched methods to the optimal residual norms attainable in the same approximation subspace. 
Both the analysis and numerical experiments reveal that the two proposed solvers behave differently depending on the residual regime: namely, sFLSQR is effective when the residual is small (small noise level), while sFLSMR is more reliable when the residual is large (large noise level).


Several directions for future work remain open. First, a deeper theoretical understanding of the behavior of sFLSMR, particularly in the presence of rapidly decaying singular values, would help explain its favorable empirical performance. Second, it would be interesting to investigate adaptive strategies for selecting the sketch dimension and for choosing the operator $\tau$ that modifies the basis vectors, for instance by adapting the truncation rank during the iterations. Finally, extending the proposed framework to incorporate additional regularization mechanisms or hybrid approaches represents a promising direction for further improving the efficiency and robustness of flexible Krylov solvers for large-scale inverse problems.

\bibliographystyle{siamplain}
\bibliography{references}
\end{document}

%% file: shared.tex
\usepackage{lipsum}
\usepackage{amsfonts}
\usepackage{graphicx}
\usepackage{epstopdf}

\usepackage{pgfplots}
\usetikzlibrary{pgfplots.groupplots}
\usepackage{tkz-euclide}
\usepackage{tikz}
\usetikzlibrary{intersections}
\usepackage{siunitx}
\usepackage{subcaption}
\usepackage{algorithm}
\usepackage{algpseudocode}
\ifpdf
  \DeclareGraphicsExtensions{.eps,.pdf,.png,.jpg}
\else
  \DeclareGraphicsExtensions{.eps}
\fi

\renewcommand{\hat}{\widehat}

\newsiamremark{remark}{Remark}
\newsiamremark{hypothesis}{Hypothesis}
\crefname{hypothesis}{Hypothesis}{Hypotheses}
\newsiamthm{claim}{Claim}
\newsiamremark{fact}{Fact}
\crefname{fact}{Fact}{Facts}

\headers{Sketched FLSQR and sketched FLSMR}{A. Bucci, S. Gazzola, and L. Robol}

\title{Randomized Flexible LSQR and LSMR\\ with applications to inverse problems\thanks{\funding{AB, SG and RB are members of the INdAM Research Group GNCS. AB is supported by the UK's Engineering and Physical Sciences Research Council (EPSRC grant EP/Z533786/1). AB also acknowledges support from the semester program “Stochastic and Randomized Algorithms in Scientific Computing: Foundations and Applications” at Institute for Computational and Experimental Research in Mathematics (ICERM), during which part of this work was completed. The work of SG was partially supported by the Italian Ministry of
University and Research (MUR) through the PRIN 2022 “Low-rank Structures and Numerical Methods in Matrix and Tensor Computations and their Application” code: 20227PCCKZ MUR
D.D. financing decree n. 104 of February 2nd, 2022 (CUP I53D23002280006). The work of 
LR was partially supported by the Italian Ministry of University and Research (MUR) through the PRIN 2022 ``MOLE: Manifold constrained Optimization and LEarning'',  code: 2022ZK5ME7 MUR D.D. financing decree n. 20428 of November 6th, 2024 (CUP I53C24002260006). The work of SG and LR was partially supported by MIUR Excellence Department Project awarded to the Department of Mathematics, University of Pisa (CUP I57G22000700001).}}}

\author{Alberto Bucci\thanks{School of Mathematics, The University of Edinburgh, Edinburgh, EH9 3FD, UK
  (\email{abucci2@ed.ac.uk}).}
\and Silvia Gazzola\thanks{Dipartimento di Matematica, Università di Pisa, Largo Bruno Pontecorvo 5, IT
  (\email{silvia.gazzola@unipi.it}, \email{leonardo.robol@unipi.it}).}
\and Leonardo Robol\footnotemark[3]}

\usepackage{amsopn}


%% file: references.bib
@string{sisc = "SIAM J. Sci. Comput."}

@article{gazzola-lowrank,
author = {Gazzola, Silvia and Meng, Chang and Nagy, James G.},
title = {{K}rylov Methods for Low-Rank Regularization},
journal = {SIAM Journal on Matrix Analysis and Applications},
volume = {41},
number = {4},
pages = {1477-1504},
year = {2020},
doi = {10.1137/19M1302727},
}

@article{LSMR,
author = {D.~C.~L.~Fong and M.~Saunders},
title = {{LSMR}: An iterative algorithm for sparse least-squares
problems},
journal = {SIAM J. Scientific Computing},
volume = {33},
pages = {2950–-2971},
year = {2011}
}

@article{martinsson2020randomized,
  title={Randomized numerical linear algebra: Foundations and algorithms},
  author={Martinsson, Per-Gunnar and Tropp, Joel A},
  journal={Acta Numerica},
  volume={29},
  pages={403--572},
  year={2020},
  publisher={Cambridge University Press}
}

@article{brown2025inner,
  title={Inner-Product Free {K}rylov Methods for Large-Scale Inverse Problems},
  author={Brown, Ariana N and Chung, Julianne and Nagy, James G and Sabat{\'e} Landman, Malena},
  journal={SIAM Journal on Scientific Computing},
  number={0},
  pages={S161--S182},
  year={2025},
  publisher={SIAM}
}

@article{FKSIRW,
issn = {1064-8275},
journal = sisc,
pages = {S47--S69},
year = {2021},
title = {Iteratively Reweighted {FGMRES} and {FLSQR} for Sparse Reconstruction},
author = {Gazzola, Silvia and Nagy, James G and Sabaté Landman, Malena}
}

@article{sabate2025randomized,
  title={Randomized and inner-product free {K}rylov methods for large-scale inverse problems},
  author={Sabat{\'e} Landman, Malena and Brown, Ariana N and Chung, Julianne and Nagy, James G},
  journal={Numerical Algorithms},
  pages={1--21},
  year={2025},
  publisher={Springer}
}

@article{chung2025randomized,
  title={Randomized {K}rylov methods for inverse problems},
  author={Chung, Julianne and Gazzola, Silvia},
  journal={arXiv preprint arXiv:2508.20269},
  year={2025}
}

@article{Saad1993,
author = {Saad, Y.},
title = {A Flexible Inner-Outer Preconditioned {GMRES} Algorithm},
journal = {SIAM J. Sci. Comput.},
volume = {14},
number = {2},
pages = {461-469},
year = {1993},
doi = {10.1137/0914028}
}

@article{sabate2025randomizednew,
title={Randomized flexible {K}rylov methods for $\ell_p$ regularization},
author = {Sabat{\'e} Landman, Malena and Nakatsukasa, Yuji},
journal={arXiv preprint arXiv:2302.13616},
  year={2025}
}

@article{chung2019,
  title={Flexible {K}rylov Methods for $\ell_p$ Regularization},
  author={Chung, J. and Gazzola, S.},
  journal=sisc,
  volume={41},
  number={5},
  pages={S149--S171},
  year={2019},
  publisher={SIAM}
}

@article{CT1,
title = {{GMRES} methods for tomographic reconstruction with an unmatched back projector},
journal = {Journal of Computational and Applied Mathematics},
volume = {413},
pages = {114352},
year = {2022},
author = {P.~C.~Hansen and K.~Hayami and K.~Morikuni},
}

@article{ABBAGMRES,
title = {{GMRES} methods for least squares problems},
journal = {SIAM J.~Matrix Anal.~Appl.},
volume = {31},
pages = {2400–-2430},
year = {2010},
author = {K.~Hayami and J.-F.~Yin and T.~Ito},
}

@article{CT2,
author = {Y.~Dong and P.~C.~Hansen and M.~E.~Hochstenbach and N.~A.~Brogaard Riis},
title = {Fixing Nonconvergence of Algebraic Iterative Reconstruction with an Unmatched Backprojector},
journal = {SIAM Journal on Scientific Computing},
volume = {41},
number = {3},
pages = {A1822-A1839},
year = {2019}
}

@article{paige1982lsqr,
  title={{LSQR}: An algorithm for sparse linear equations and sparse least squares},
  author={Paige, Christopher C and Saunders, Michael A},
  journal={ACM Transactions on Mathematical Software (TOMS)},
  volume={8},
  number={1},
  pages={43--71},
  year={1982},
  publisher={ACM New York, NY, USA}
}

@article{HnPlSt09,
author = "I.~Hn{\v{e}}tynkov{\'a} and M.~Ple{\v{s}}inger and Z.~Strakos",
title = "The regularizing effect of the {G}olub-{K}ahan iterative bidiagonalization and revealing the noise level in the data",
journal = "BIT",
volume = 49,
year = 2009,
pages = "669--696"
}

@article{GazzolaSilvia2016Iotd,
journal = {BIT Numerical Mathematics},
pages = {893--918},
volume = {56},
number = {3},
year = {2016},
title = {Inheritance of the discrete {P}icard condition in {K}rylov subspace methods},
author = {Gazzola, Silvia and Novati, Paolo}
}

@article{hmt, title={Finding Structure with Randomness: Probabilistic Algorithms for Constructing Approximate Matrix Decompositions}, volume={53}, DOI={10.1137/090771806}, number={2}, journal={SIAM Rev.}, author={Halko, N. and Martinsson, Per-Gunnar and Tropp, Joel A.}, year={2011}, pages={217–288}}

@article{nakatsukasa2024fast,
  title={Fast and accurate randomized algorithms for linear systems and eigenvalue problems},
  author={Nakatsukasa, Yuji and Tropp, Joel A},
  journal={SIAM Journal on Matrix Analysis and Applications},
  volume={45},
  number={2},
  pages={1183--1214},
  year={2024},
  publisher={SIAM}
}

@article{mahoney2011randomized,
  title={Randomized algorithms for matrices and data},
  author={Mahoney, Michael W and others},
  journal={Foundations and Trends{\textregistered} in Machine Learning},
  volume={3},
  number={2},
  pages={123--224},
  year={2011},
  publisher={Now Publishers, Inc.}
}

@article{woodruff2014sketching,
  title={Sketching as a tool for numerical linear algebra},
  author={Woodruff, David P and others},
  journal={Foundations and Trends{\textregistered} in Theoretical Computer Science},
  volume={10},
  number={1--2},
  pages={1--157},
  year={2014},
  publisher={Now Publishers, Inc.}
}

@article{balabanov2022randomized,
  title={Randomized {G}ram--{S}chmidt process with application to {GMRES}},
  author={Balabanov, Oleg and Grigori, Laura},
  journal={SIAM Journal on Scientific Computing},
  volume={44},
  number={3},
  pages={A1450--A1474},
  year={2022},
  publisher={SIAM}
}

@article{sadok1999cmrh,
  title={{CMRH}: A new method for solving nonsymmetric linear systems based on the {H}essenberg reduction algorithm},
  author={Sadok, Hassane},
  journal={Numerical Algorithms},
  volume={20},
  number={4},
  pages={303--321},
  year={1999},
  publisher={Springer}
}

@article{brown2025h,
  title={{H-CMRH}: An inner product free hybrid {K}rylov method for large-scale inverse problems},
  author={Brown, Ariana N and Sabat{\'e} Landman, Malena and Nagy, James G},
  journal={SIAM Journal on Matrix Analysis and Applications},
  volume={46},
  number={1},
  pages={232--255},
  year={2025},
  publisher={SIAM}
}

@book{hansen2010discrete,
  title={Discrete Inverse Problems: Insight and Algorithms},
  author={P.~C.~Hansen},
  year={2010},
  publisher={SIAM}
}

@article{ChungPalmer2015,
  author    = {Julianne Chung and Katrina Palmer},
  title     = {A Hybrid {LSMR} Algorithm for Large-Scale {T}ikhonov Regularization},
  journal   = {SIAM Journal on Scientific Computing},
  volume    = {37},
  number    = {5},
  pages     = {S562--S580},
  year      = {2015},
  doi       = {10.1137/140975024}
}

@article{survey,
journal = {SIAM Review},
pages = {205–284},
volume = {66},
number = {2},
year = {2024},
title = {Computational Methods for Large-Scale
Inverse Problems: A Survey on Hybrid
Projection Methods},
author = {J.~Chung and S.~Gazzola}
}

@article{astra,
  title={The {ASTRA} Toolbox: A platform for advanced algorithm development in electron tomography},
  author={Van Aarle, Wim and Palenstijn, Willem Jan and De Beenhouwer, Jan and Altantzis, Thomas and Bals, Sara and Batenburg, K Joost and Sijbers, Jan},
  journal={Ultramicroscopy},
  volume={157},
  pages={35--47},
  year={2015},
  publisher={Elsevier}
}

@article{gazzola2015survey,
author = "S.~Gazzola and P.~Novati and M.~R.~Russo",
title = "{O}n {K}rylov projection methods and {T}ikhonov regularization",
journal = "Electron.~Trans.~Numer.~Anal.",
volume = 44,
pages = "83--123",
year = 2015
}
